\documentclass[12pt]{amsart}
\usepackage{amssymb}
\usepackage{amscd}
\usepackage{amsmath}
\theoremstyle{plain}
\newtheorem{Thm}{Theorem}[section]
\newtheorem{Thm*}{Theorem}[section]
\newtheorem{Thm'}[Thm]{``Theorem"}
\newtheorem{Cor}[Thm]{Corollary}

\newtheorem{Prop}[Thm]{Proposition}
\newtheorem{Lem}[Thm]{Lemma}
\newtheorem{Cl}[Thm]{Claim}

\theoremstyle{definition}

\newtheorem{Emp}[Thm]{}

\newtheorem{Q}[Thm]{Question}
\numberwithin{equation}{section}
\newcommand{\pp}{\boxtimes}
\newcommand{\ov}{\overline}
\newcommand{\isom}{\overset {\thicksim}{\to}}
\newcommand{\bisom}{\overset {\thicksim}{\leftarrow}}
\newcommand{\B}[1]{\mathbb#1}
\newcommand{\cal}[1]{\mathcal{#1}}
\newcommand{\C}[1]{\cal#1}
\newcommand{\Om}{\Omega}
\newcommand{\Si}{\Sigma}
\newcommand{\om}{\omega}
\newcommand{\surj}{\twoheadrightarrow}
\newcommand{\lra}{\longrightarrow}
\newcommand{\lla}{\longleftarrow}
\newcommand{\hra}{\hookrightarrow}
\newcommand{\wt}{\widetilde}
\newcommand{\Gm}{\Gamma}

\newcommand{\G}{\mathbf{G}}
\newcommand{\D}{\mathbf{D}}
\newcommand{\bb}{\mathbf{B}}
\renewcommand{\P}{\mathbf{P}}
\newcommand{\M}{\mathbf{M}}
\newcommand{\X}{\mathbf{X}}
\newcommand{\Y}{\mathbf{Y}}
\newcommand{\Z}{\mathbf{Z}}
\newcommand{\U}{\mathbf{U}}
\newcommand{\V}{\mathbf{V}}
\renewcommand{\H}{\mathbf{H}}
\renewcommand{\c}{\mathbf{c}}
\newcommand{\K}{\mathbf{K}}
\newcommand{\T}{\mathbf{T}}
\renewcommand{\S}{\mathbf{S}}
\renewcommand{\L}{\mathbf{L}}

\newcommand{\gm}{\gamma}
\newcommand{\dt}{\delta}
\newcommand{\Dt}{\Delta}

\newcommand{\bs}{\backslash}

\newcommand{\m}{^{\times}}

\newcommand{\un}[1]{\underline{#1}}

\newcommand{\al}{\alpha}

\newcommand{\la}{\lambda}

\newcommand{\form}[1]{(\ref{Eq:#1})}

\newcommand{\rl}[1]{Lemma \ref{L:#1}}

\newcommand{\rcl}[1]{Claim \ref{C:#1}}
\newcommand{\rp}[1]{Proposition \ref{P:#1}}

\newcommand{\re}[1]{\ref{E:#1}}
\newcommand{\rco}[1]{Corollary \ref{C:#1}}

\newcommand{\rt}[1] {Theorem \ref{T:#1}}

\newcommand{\sm}{\smallsetminus}

\newcommand{\pr}{\operatorname{pr}}

\newcommand{\im}{\operatorname{Im}}
\newcommand{\unip}{\operatorname{un}}

\newcommand{\Spec}{\operatorname{Spec}}

\newcommand{\Res}{\operatorname{Res}}

\newcommand{\st}{\operatorname{st}}

\newcommand{\Ad}{\operatorname{Ad}}

\newcommand{\Stab}{\operatorname{Stab}}

\newcommand{\Tr}{\operatorname{Tr}}
\renewcommand{\sc}{\operatorname{sc}}

\newcommand{\res}{\operatorname{res}}

\newcommand{\Lie}{\operatorname{Lie}}

\newcommand{\Id}{\operatorname{Id}}
\newcommand{\der}{\operatorname{der}}

\newcommand{\Rep}{\operatorname{Rep}}

\newcommand{\cl}{\operatorname{cl}}

\newcommand{\End}{\operatorname{End}}

\newcommand{\Ind}{\operatorname{Ind}}

\newcommand{\sat}{\operatorname{sat}}

\newcommand{\Hom}{\operatorname{Hom}}

\newcommand{\reg}{\operatorname{reg}}
\newcommand{\rss}{\operatorname{rss}}

\newcommand{\wh}{\widehat}

\newcommand{\GL}{\mathbf{GL}}

\begin{document}

\title[Geometric approach to parabolic induction]
{Geometric approach to parabolic induction}

%\maketitle

%\author{Roman Bezrukavnikov}
%\address{Department of Mathematics\\
%Massachusetts Institute of Technology\\
%77 Massachusetts Avenue\\
%Cambridge, MA 02139, USA} \email{bezrukav@math.mit.edu}
%\date{\today}

\author{David Kazhdan}
\address{Institute of Mathematics\\
The Hebrew University of Jerusalem\\
Givat-Ram, Jerusalem,  91904\\
Israel} \email{kazhdan@math.huji.ac.il}

\author{Yakov Varshavsky}

\address{Institute of Mathematics\\
The Hebrew University of Jerusalem\\
Givat-Ram, Jerusalem,  91904\\
Israel} \email{vyakov@math.huji.ac.il}


%\thanks{Address: Institute of Mathematics,
%the Hebrew University of Jerusalem, Givat-Ram, Jerusalem,  91904,
%Israel; e-mail: kazhdan@math.huji.ac.il, vyakov@math.huji.ac.il}

\thanks{The project have received funding from ERC under grant agreement No 669655, Y.V. was partially 
supported by the ISF grant 1017/13.}
%\abstract
\begin{abstract}
In this note we construct a ``restriction" map from the cocenter of a reductive group $G$ over a local
non-archimedean field $F$ to the cocenter of a Levi subgroup. We show that the dual map corresponds to parabolic induction and
deduce that parabolic induction preserves stability. We also give a new (purely geometric) proof that the character of normalized parabolic induction does not depend on the parabolic subgroup. In the appendix, we use a similar argument to extend a theorem of Lusztig--Spaltenstein on induced unipotent classes to all infinite fields. We also prove a group version of a theorem of Harish-Chandra about the density of the span of regular semisimple orbital integrals.
\end{abstract}
\maketitle

\centerline{\em To Iosif Bernstein with gratitude and best wishes on his birthday}

\tableofcontents

\section*{Introduction}
\noindent Let $\G$ be a linear algebraic group over a local non-archimedean field $F$, and let $G=\G(F)$.  We denote by $\C{H}(G)$ the space of smooth measures with compact support on $G$, by $\C{H}(G)_{G}$ the space of the coinvariants of $\C{H}(G)$ with respect to the adjoint action of $G$, and let $\wh{C}^{G}(G):=\Hom_{\B{C}}(\C{H}(G)_{G},\B{C})$ be the space of invariant generalized functions on $G$.
To every admissible representation $\pi$ of $G$, one can associate its character $\chi_{\pi}\in\wh{C}^{G}(G)$.

Now assume that $\G$ is connected reductive, $\P\subset \G$ a parabolic subgroup, $\M\subset\P$ a Levi subgroup,
and $\U\subset\P$ the unipotent radical. For every admissible representation $\rho$ of $M=\M(F)$, we denote by
$\pi=\un{i}_{P;M}^G(\rho)$ the admissible representation of $G$ obtained by the normalized parabolic induction of $\rho$.

In this work we construct a continuous map $i_{P;M}^G:\wh{C}^{M}(M)\to \wh{C}^{G}(G)$, satisfying
$i_{P;M}^G(\chi_{\rho})=\chi_{\pi}$ for every $\rho$ as above. Namely,
we construct a ``restriction" map $r_{P;M}^G:\C{H}(G)_{G}\to \C{H}(M)_{M}$, and define $i_{P;M}^G$ to be the linear dual of
$r_{P;M}^G$.

Then we show that $i_{P;M}^G$ does not depend on $\P\supset\M$. From this
we conclude that for each $\rho$ as above, the set of composition factors of $\un{i}_{P;M}^G(\rho)$
does not depend on $\P$. This gives a geometric proof of \cite[Lem 5.4 (iii)]{BDK}
\footnote{It was shown in \cite[Thm B]{Ka2} that the span of characters of smooth irreducible representations of $G$
is dense in $\wh{C}^{G}(G)$. Therefore the independence assertion for $i_{P;M}^G$
follows from the corresponding result for characters (\cite[Lem 5.4 (iii)]{BDK}).}.

Finally, we show that $i_{P;M}^G$ preserves stability. This implies that parabolic induction of
a stable representation is stable. This result is considered to be well-known by specialists,
but does not seem to appear in a written form.

%Notice that all of our arguments are completely geometric.
%In particular, the field of coefficients can be taken any field of characteristic zero.

In the first appendix, we study a related question, motivated by the work of Lusztig--Spaltenstein
\cite{LS}.

Let $\G$ and $\P$, $\M$ and $\U$ be as above, but over an arbitrary infinite field $F$.
To every unipotent conjugacy class $C\subset M$ we associate an $\Ad G$-invariant subset
$C_{P;G}:=\cup_{g\in G}g(C\cdot U)g^{-1}\subset G$. Then $C_{P;G}$ is a union of unipotent conjugacy classes in $G$.
A natural question is to what extent the set $C_{P;G}$  depends on $\P\supset\M$.

In their work, Lusztig and Spaltenstein \cite{LS} showed, using representation theory, that the Zariski closure of $C_{P;G}$ does not depend on $\P$. A simpler proof of this fact was given later by Lusztig \cite[Lem 10.3(a)]{Lu}. This result can be thought of as the assertion that $C_{P;G}$ is ``essentially"  independent of $\P$, if $F$ is algebraically closed.

We extend this result to an arbitrary $F$. Namely, for every algebraic variety
$\X$ over $F$ and a subset $A\subset X=\X(F)$ we define a ``saturation" $\sat(A)\subset X$. The main result of the appendix
asserts that the saturation $\sat(C_{P;G})\subset G$ does not depend on $\P$.

Since every Zariski closed subset is
saturated, our result is an extension of the theorem of Lusztig--Spaltenstein.
Similarly, when $F$ is a local field, our result implies that the closure of
$C_{P;G}$ in the analytic topology does not depend on $\P$.

The result in the appendix indicates that the independence from $\P$ of normalized parabolic induction has a purely algebraic flavour.
We also believe that the notion of saturation is interesting in its own right and deserves to be studied.

In the second appendix we give a proof of a group version of a theorem of Harish-Chandra, asserting that in
characteristic  zero, the span of regular semisimple orbital integrals is dense in $\wh{C}^{G}(G)$.
Though this very important result is considered to be well-known among specialists, only its Lie algebra analog \cite[Thm 3.1]{HC}
seems to appear in the literature. In particular, this fills a small gap in the argument of \cite[Thm B]{Ka2}.

Our paper is organized as follows. In Section 1 we describe general properties
of reductive groups and the so-called  generalized Grothendieck--Springer resolutions.
In Section 2 we study non-vanishing top degree differential forms, which are basic tools for this work.
Note that in these two sections the ground field $F$ is arbitrary, while starting from
 Section 3 the field $F$ is local non-archimedean.

In Section 3 we introduce smooth measures with compact
support and carry out our construction of the restriction $r_{P;M}^G$ and the induction $i_{P;M}^G$.
In Section 4 we show that the induction map sends a character of a representation to a character of the induced representation. Then in Section 5 we construct another restriction map $R_H^G$, defined
for a connected equal rank subgroup $\H$ of $\G$, and show that $R_H^G$ is compatible with $r_{P;M}^G$. Finally, in Section 6 we deduce from the results of Section 5 that the normalized parabolic induction is independent of $\P$ and preserves stability.

Notice that in Sections 1--5 we work with non-normalized induction, which has a purely geometric interpretation,
while we pass to normalized induction only in Section 6.

We thank Joseph Bernstein, Yuval Flicker and the referee for a number of helpful suggestions and corrections.

\section{Preliminaries on algebraic groups}

%Let $F$ be an arbitrary field. All algebraic varieties and morphisms of algebraic varieties will be over $F$.

\begin{Emp} \label{E:chev}
{\bf The Chevalley map.}
(a) Let $\G$ be a linear algebraic group over a field $F$, let $\c_\G:=\Spec F[\G]^\G$ be {\em the Chevalley space} of $\G$, and let $\nu_\G:\G\to \c_\G$ be the morphism dual to the inclusion
$F[\G]^\G\hra F[\G]$. %called {\em the Chevalley map}.

(b) A homomorphism of linear algebraic groups $\H\to\G$ induces a morphism $\pi_{\H,\G}:\c_\H\to\c_\G$ of the Chevalley spaces, making the following diagram commutative
\[
\begin{CD}
\H @>\nu_{\H}>>\c_\H\\
@VVV @VV\pi_{\H,\G} V\\
\G @>\nu_\G>> \c_\G.
\end{CD}
\]

(c) Let $\G$ be connected reductive and split, $\T\subset \G$  a maximal split torus, and
$W_{\G}=W_{\G,\T}$ the Weyl group of $\G$. Then the restriction $\nu_{\G}|_{\T}:\T\to\c_{\G}$ is surjective and induces an isomorphism
$W_{\G}\bs\T\isom \c_{\G}$ (see \cite[Cor 6.4]{St}).
\end{Emp}

\begin{Emp} \label{E:reg}
{\bf Notation.} Let $\G$ be as in \re{chev}, and let $\H\subset \G$ be a closed subgroup.

(a) Let $\la_{\G}:\G\to\B{G}_m$ be the homomorphism $g\mapsto\det\Ad g$ and let $\Dt_{\H,\G}\in F[\H]^\H=F[\c_\H]$ be the $\Ad \H$-invariant function  $h\mapsto\det(\Ad h^{-1}-1,\Lie \G/\Lie \H)$.

(b) Let $\H^{\reg/\G}\subset \H$ (resp. $\c_\H^{\reg/\G}\subset \c_\H$) be the open subscheme defined by the condition $\Dt_{\H,\G}\neq 0$. By construction, $\H^{\reg/\G}=\nu_\H^{-1}(\c_\H^{\reg/\G})$.

(c) Let $\K\subset \H$ be a closed subgroup. Then $\Dt_{\K,\G}=\Dt_{\K,\H}\cdot (\Dt_{\H,\G}|_\K)$, thus
$\K^{\reg/\G}=\K^{\reg/\H}\cap \H^{\reg/\G}$.
\end{Emp}

\begin{Emp}
{\bf Remark.} Notice that subset $\H^{\reg/\G}\subset\H$ should not be confused with
a more standard subset $\H^{\G-\reg}:=\H\cap \G^{\reg}\subset\H$.
\end{Emp}

\begin{Emp} \label{E:grspr}
{\bf The generalized Grothendieck--Springer resolution.}

(a) Let $\H$ (see \re{reg}) act on an algebraic variety $\X$ over $F$. We denote by $\Ind_\H^\G(\X):=\G\overset{\H}{\times}\X$ the quotient
of $\G\times \X$ by $\H$, acting by $h(g,x):=(gh^{-1},h(x))$. Then $\Ind_\H^\G(\X)$ is equipped with an action of $\G$,  given by $g'([g,h])=[g'g,h]$.

(b) Assume that $\H$ acts on itself  by conjugation, and set $\wt{\G}_\H:=\Ind_\H^\G(\H)$.
%and call it the (generalized) Grothendieck--Springer resolution.
Explicitly, $\wt{\G}_\H=\G\overset{\H,\Ad}{\times}\H$, where $\Ad$ indicates the adjoint action.

(c) We have a natural closed embedding  $(\pr_{\G/\H},a_{\H,\G}):\wt{\G}_\H\hra (\G/\H)\times \G$, defined by $[g,h]\mapsto ([g],ghg^{-1})$, whose image consists of all pairs  $([g],x)$ such that for every representative $g\in\G$  of $[g]$ we have $x\in \H_{g}:=g\H g^{-1}$. In particular, the projection $\pr_{\G/\H}:\wt{\G}_\H\to \G/\H$ is smooth, and for every $[g]\in \G/\H$, the map $a_{\H,\G}:\wt{\G}_\H\to\G$ identifies the
fiber $\pr_{\G/\H}^{-1}([g])$ with $\H_g$.
 \end{Emp}
%Let $\iota:\G/\H\to \wt{\G}_\H$ be the unit section $[g]\mapsto [g,1]$.

\begin{Emp} \label{E:simple}
{\bf Simple properties.}
(a) Every $\Ad \H$-invariant morphism $f:\H\to \X$ gives rise to an $\Ad \G$-equivariant morphism
$\wt{\G}_\H\to \X:[g,h]\mapsto f(h)$. In particular, $\la_\H$ gives rise to a morphism $\la_{\wt{\G}_\H}:\wt{\G}_\H\to\B{G}_m$, and $\Dt_{\H,\G}$ gives rise to a function $\Dt_{\wt{\G}_\H}:\wt{\G}_\H\to\B{A}^1$.

(b) By (a), the morphism $\nu_\H:\H\to \c_\H$ gives rise to a morphism $\nu_{\wt{\G}_\H}:\wt{\G}_\H\to \c_\H$. Moreover, by \re{chev}(b), we have a commutative diagram
\[
\begin{CD}
\wt{\G}_\H @>\nu_{\wt{\G}_\H}>>\c_\H\\
@Va_{\H,\G} VV @VV\pi_{\H,\G} V\\
\G @>\nu_\G>> \c_\G,
\end{CD}
\]
which induces a morphism $\wt{\G}_\H\to \G\times_{\c_\G}\c_\H$.

(c) Set  $\wt{\G}^{\reg}_\H:=\nu_{\wt{\G}_\H}^{-1}(\c^{\reg/\G}_\H)\subset\wt{\G}_\H$. Explicitly,
$\wt{\G}_\H^{\reg}=\Ind_\H^\G(\H^{\reg/\G})$. By definition, the morphism of (b) induces a morphism
$\iota_{\H,\G}:\wt{\G}^{\reg}_\H\to \G\times_{\c_\G}\c^{\reg/\G}_\H$.

(d) Let $\K\subset \H$ be a closed subgroup. Then the morphism
$a_{\K,\H}:\wt{\H}_\K\to \H$ induces a morphism $\wt{\G}_\K=\Ind_\H^\G(\wt{\H}_\K)\to \Ind_\H^\G(\H)=\wt{\G}_\H$, which we again denote by $a_{\K,\H}$.
Note that the composition $a_{\H,\G}\circ a_{\K,\H}:\wt{\G}_\K\to \wt{\G}_\H\to\G$ is equal to  $a_{\K,\G}$.  By \re{reg}(c), $a_{\K,\H}$ induces a morphism
$\wt{\G}^{\reg}_\K\to\wt{\G}^{\reg}_\H$.

(e) We also set $\wt{\H}_\K^{\reg/\G}:=\nu_{\wt{\H}_\K}^{-1}(\c_\K^{\reg/\G})\subset \wt{\H}_\K$.
\end{Emp}

%\begin{Emp}
%{\bf The simply connected case.}
%Assume that $\G$ is semisimple and simply connected. Then the homomorphism $\la_{\P}$ has a unique
%square-root $|\cdot|^{1/2}_{\P}:\P\to \B{G}_m$. Indeed, let $\U_\P\subset \P$ be the unipotent radical, and
%$M:=P/U_P$ be the reductive quotient of $\P$, then $|\cdot|_P$ corresponds to the character $2\rho_G-2\rho_{M}$, thus $|\cdot|_P$ corresponds to $\rho_G-\rho_{M}$, which is indeed a character of $\P$, since
%In particular, in this case we can consider  another differential form
%$\om^m_{\wt{G}}(v):=\Dt^{1/2}_P\om_{\wt{G}}^r$ on $\wt{G}$, where $(\cdot)^m$ stays for ``middle".
%Notice that $\om^m_{\wt{G}}(v)$ can be characterized as the unique extension of $\om(v)$, invariant under the map $g\mapsto g^{-1}$.
%\end{Emp}

\begin{Emp} \label{E:eqrk}
{\bf The equal rank case.}
(a) Let $\G$ be connected reductive, and let $\H\subset \G$ be a connected {\em equal rank} subgroup, by which we mean that a maximal torus $\T\subset \H$ is a maximal torus of $\G$. Let $\U=\U_\H\subset \H$ be the unipotent radical of $\H$, and fix a maximal torus $\T\subset \H$, defined over $F$.

(b) Assume that $\T$ is split. Then, by definition, an element $t\in \T$ belongs to $\H^{\reg/\G}$ if and only if $\al(t)\neq 0$ for every root $\al\in\Phi(\G,\T)\sm \Phi(\H,\T)$
\footnote{For every closed subgroup $\L\subset\G$ normalized by $\T$, we denote by $\Phi(\L,\T)\subset\Phi(\G,\T)$ the set of non-zero weights in
$\Lie L\subset\Lie G$.}. Equivalently, this happens if and only if the connected stabilizer $\Z_\G(t)^0$  equals $\Z_\H(t)^0$.

(c) If, in addition, the derived group of $\G$ is simply connected, then each stabilizer $\Z_\G(t)$ is connected. Thus for every $t\in \T\cap \H^{\reg/\G}$, we have $\Z_\G(t)=\Z_\H(t)\subset \H$.
\end{Emp}

\begin{Cl} \label{C:eqrk}
In the case of \re{eqrk} (a), the group $\H$ has a Levi subgroup (see \cite[11.22]{Bo}). Moreover, there exists a unique Levi subgroup $\M\subset\H$, containing $\T$.
\end{Cl}

\begin{proof}
By uniqueness, we can extend scalars to a finite separable extension, thus assuming that
$\T$ is split. In this case, the subgroup $\H\subset\G$ is generated by $\T$ and the root subgroups $\U_{\al}\subset\G$ for $\al\in\Phi(\H,\T)$.
Indeed, this follows from the fact that $\H$ is generated by its Borel subgroups $\bb\supset\T$, and that the corresponding assertion for
solvable $\H$ follows from \cite[Prop 14.4]{Bo}.

Next we observe that the set $\Phi(\U,\T)$ consists of all $\al\in\Phi(\H,\T)$ such that $-\al\notin \Phi(\H,\T)$.
Therefore a Levi subgroup $\M\supset\T$ of $\H$ has to coincide with the subgroup $\M'\subset \G$ generated  by $\T$ and the roots subgroups $\U_{\al}\subset\G$ for all $\al\in\Phi':=\Phi(\H,\T)\sm \Phi(\U,\T)$.

It remains to show that the subgroup $\M'\subset \H$ defined in the previous paragraph is indeed a Levi subgroup. Since $\M'\U=\H$, it suffices to show that
$\Phi(\M',\T)\subset\Phi'$. Recall that $\G$ has an automorphism $\iota$  such that $\iota(\T)=\T$ and $\iota(\U_{\al})=\U_{-\al}$
for every $\al\in\Phi(\G,\T)$. Since the subset $\Phi'\subset\Phi(\G,\T)$ is stable under the map $\al\mapsto-\al$, we conclude that
$\M'\subset\H\cap\iota(\H)$, thus $\Phi(\M',\T)\subset\Phi(\H,\T)\cap \Phi(\iota(\H),\T)=\Phi'$.
\end{proof}

\begin{Emp}
{\bf Remark.}
In this work we only consider the case when $\H$ is either a parabolic subgroup of $\G$ or a connected centralizer of a semisimple element.
In these cases, \rcl{eqrk} is well-known (see \cite[Cor 14.19]{Bo}). On the other hand, we believe
that the context of equal rank subgroups is the ``correct framework" to work in.
\end{Emp}

\begin{Lem} \label{L:eqrk}
In the notation  of \rcl{eqrk}, let $p$ be the projection $\H\to \H/\U\cong \M$, and let $i$
 be the inclusion $\M\hra \H$.

(a) Set $\H':=p^{-1}(\M^{\reg/\G})\subset\H$.  Then the morphism $\Ind_\M^\H(\M^{\reg/\G})\to\H'$, induced by $a_{\M,\H}:\wt{\H}_\M\to\H$,
is an isomorphism.

(b)  The morphisms $\pi_{\H,\M}:\c_\H\to \c_\M$ and  $\pi_{\M,\H}:\c_\M\to \c_\H$, induced by $p$ and $i$ respectively (see \re{chev}(b)), are isomorphisms.

(c) The isomorphism $\pi_{\M,\H}:\c_\M\isom \c_\H$ induces an isomorphism $\c^{\reg/\G}_\M\isom \c^{\reg/\G}_\H$.

(d) We have the equality $p^{-1}(\M^{\reg/\G})=\H^{\reg/\G}$.

(e) The morphisms $\Ind_\M^\H(\M^{\reg/\G})\to \H^{\reg/\G}$ and $\wt{\G}^{\reg}_\M\to \wt{\G}^{\reg}_\H$ induced by $a_{\M,\H}$ (compare  \re{simple}(d)) are isomorphisms.

(f) Set $e_{\G}:=\nu_\G(1)\in\c_\G$ and similarly for $\H$. Then the morphism $\pi_{\H,\G}:\c_\H\to\c_\G$ satisfies $\pi_{\H,\G}^{-1}(e_\G)=e_{\H}$.
\end{Lem}

\begin{proof}
 (a) Since $\H=\U\rtimes \M$, it suffices to show that for every $m\in \M^{\reg/\G}$ the map $f_m:\U\to \U:u\mapsto m^{-1}umu^{-1}$ is an isomorphism.

Consider the upper (or lower) filtration  $\U^{(i)}$ of $\U$. It remains to show that $f_m$ induces an isomorphism on each quotient $\U^{(i)}/\U^{(i+1)}$. Using the natural isomorphism $\U^{(i)}/\U^{(i+1)}\cong \Lie \U^{(i)}/\U^{(i+1)}$, it suffices to show that $\Ad(m^{-1})-\Id$ is invertible on $\Lie \U^{(i)}/\U^{(i+1)}$.
Since $m\in \M^{\reg/\G}$, the map $\Ad(m^{-1})-\Id$ is invertible on $\Lie \G/\Lie \M$. Hence it is invertible on $\Lie \U=\Lie \H/\Lie \M$,
and the assertion follows.

(b) Since $p\circ i=\Id_{\M}$, it suffices to show that the pullback $i^*:F[\H]^{\H}\to F[\M]^\M$ is injective.
By (a), the induced map $i^*:F[\H']^{\H}\to F[\M^{\reg/\G}]^\M$ is injective. Since $\H'\subset\H$
is Zariski dense,  the restriction map $F[\H]\to F[\H']$ is injective, and the assertion follows.

(c) Extending scalars to a finite separable extension, we can assume that $\T$ is split. Then, by \re{chev}(c), it remains to show that $\T\cap \M^{\reg/\G}=\T\cap \H^{\reg/\G}$. But this follows from the explicit description of both sides, given in \re{eqrk}(b).

The assertion (d) follows from (c), while (e)  follows from (a) and (d).

(f)  By (b), we can replace $\H$ by $\M$, thus assuming that $\H$ is reductive. Extending scalars, we can assume that $\T$ is split. In this case, $\pi_{\H,\G}:\c_\H\to \c_\G$ is the projection $W_\H\bs \T\to W_\G\bs \T$ (see \re{chev}(c)), and the assertion is immediate.
\end{proof}

\begin{Lem} \label{L:eqrank}
In the case of \re{eqrk}(a), assume that the derived group of $\G$ is simply connected. Then

(a) the morphism $\pi_{\H,\G}:\c_\H^{\reg/\G}\to \c_\G$ is \'etale.

(b) the morphism   $\iota_{\H,\G}:\wt{\G}^{\reg}_\H\to \G\times_{\c_\G}\c^{\reg/\G}_\H$
from \re{simple}(c) is an isomorphism.
\end{Lem}

\begin{proof}
Extending scalars to a finite separable extension of $F$, we can assume that $\T$ is split.

(a) By \rl{eqrk}(c), we can replace $\H$ by its Levi subgroup $\M$, thus we can assume that $\H$ is reductive.
Then $\pi_{\H,\G}:\c_\H\to \c_\G$ is simply the projection $W_\H\bs \T\to W_\G\bs \T$ (see \re{chev}(c)). Thus it remains to show that for every $t\in \T\cap \H^{\reg/\G}$, we have the equality
of stabilizers $\Stab_{W_\H}(t)= \Stab_{W_\G}(t)$. But our assumption on $\G$ implies that $\Z_\H(t)=\Z_\G(t)$ (see \re{eqrk}(c)),
so the assertion follows.

(b) Note that $\wt{\G}^{\reg}_\H$ is \'etale over $\G$ by \rco{etale} below, while $\G\times_{\c_\G}\c^{\reg/\G}_\H$ is \'etale over $\G$ by (a).
Hence $\iota_{\H,\G}$ is \'etale. Thus, in order to show that  $\iota_{\H,\G}$ is an isomorphism, it suffices to show that $\iota_{\H,\G}$ is
a bijection (on $\ov{F}$-points).

First we show that $\iota_{\H,\G}$ is surjective. Since $\nu_\H|_\T:\T\to \c_\H$ is surjective,
every element of $\G\times_{\c_\G}\c^{\reg/\G}_\H$ has a form $(g,\nu_\H(t))$ for some $g\in \G$ and $t\in \T\cap \H^{\reg/\G}$ such that
$\nu_\G(g)=\nu_\G(t)$. Let $g=su$ be the Jordan decomposition. Then $\nu_\G(s)=\nu_\G(g)=\nu_\G(t)$ (see \cite[Cor 6.5]{St}). Since $s$ and $t$ are semisimple, they are $\G$-conjugate (by \cite[Cor 6.6]{St}). Since $\iota_{\G,\H}$ is $\G$-equivariant, we can replace $g$ by its conjugate, thus assuming that $s=t\in\T\cap \H^{\reg/\G}$.

Since $us=su$, we get that $u\in \Z_\G(s)\subset\H$ (by \re{eqrk}(c)). Hence $g=su\in\H$, and $\nu_\H(g)=\nu_\H(s)\in \c_\H^{\reg/\G}$.
Thus $g\in \H^{\reg/\G}$, and $\iota_{\H,\G}([1,g])=(g,\nu_\H(t))$.

To show the injectivity, assume that two elements $\wt{g}=[g,h]$ and $\wt{g}'=[g',h']$ of $\wt{\G}_\H^{\reg}$ satisfy $\iota_{\H,\G}(\wt{g})=\iota_{\H,\G}(\wt{g}')$. Since $\iota_{\G,\H}$ is $\G$-equivariant, we can replace $\wt{g}$ and $\wt{g}'$ by their  $g'^{-1}$-conjugates, thus assuming that $g'=1$. In this case, the identity $\iota_{\H,\G}(\wt{g})=\iota_{\H,\G}(\wt{g}')$ implies that  $\nu_\H(h)=\nu_\H(h')$ and $ghg^{-1}=h'$. It suffices to show that $g\in \H$, hence $\wt{g}=[1,ghg^{-1}]=\wt{g}'$.

Let $h=su$ and $h'=s'u'$ be the Jordan decompositions. Then $gsg^{-1}=s'$ and $\nu_\H(s)=\nu_\H(s')$. Thus $s$ and $s'$ are $\H$-conjugates.
Hence, replacing $(g,h)$ by $(gx^{-1},xhx^{-1})$ for some $x\in \H$, we may assume that $s=s'$, thus $g\in \Z_\G(s)$.
Since $h\in \H^{\reg/\G}$, we get $s\in \H^{\reg/\G}$. By  \re{eqrk}(c), we conclude that $\Z_\G(s)\subset \H$, thus
$g\in \H$.
\end{proof}

\begin{Cor} \label{C:eqrank}
In the case of \re{eqrk}, the morphism   $\iota_{\H,\G}:\wt{\G}^{\reg}_\H\to \G\times_{\c_\G}\c^{\reg/\G}_\H$
from \re{simple}(c) is finite and surjective.
\end{Cor}

\begin{proof}
Let $\G^{\sc}$ be the simply connected covering of the derived group of $\G$. Consider the natural isogeny
$\pi:\G':=\G^{\sc}\times \Z(\G)^0\to \G$, and set $\H':=\pi^{-1}(\H)\subset \G'$. Then we have a commutative diagram
\[
\begin{CD}
\wt{\G'}^{\reg}_{\H'} @>\iota_{\H',\G'}>> \G'\times_{\c_{\G'}}\c^{\reg/{\G'}}_{\H'}\\
@VVV @VVV \\
\wt{\G}^{\reg}_\H @>\iota_{\H,\G}>> \G\times_{\c_\G}\c^{\reg/\G}_\H,
\end{CD}
\]
where vertical arrows are finite surjective morphisms, induced by $\pi$. Now, $\iota_{\H',\G'}$ is an isomorphism by \rl{eqrank}.
Therefore $\iota_{\H,\G}$ is finite and surjective.
\end{proof}

\begin{Emp}
{\bf Remark.} Although the morphism $\iota_{\H,\G}$ from \rco{eqrank} is not an isomorphism in general, it is an isomorphism over a ``strongly regular locus".
Moreover, the whole morphism $\iota_{\H,\G}$ ``can be made an isomorphism", if one replaces the (singular) Chevalley spaces $\c_\G$ and $\c_\H$ by their smooth Artin stack versions. %$[\Ker\pi\bs \c_{\G'}]$  and $[\Ker\pi\bs \c_{\H'}]$.
%(b) The map $\iota_{\H,\G}$ is proper in much greater generality.
\end{Emp}

\section{Top degree differential forms}

\begin{Emp} \label{E:Om}
{\bf Notation.} (a) For every smooth algebraic (or analytic) variety $\X$ over $F$ of dimension $d$, we denote by
$\C{K}_\X=\Om^{d}_\X$ the canonical bundle on $\X$, that is, the line bundle of top degree differential forms on $\X$.

(b) For every map of smooth algebraic (or analytic) varieties $f:\X\to \Y$ over $F$ of dimension $d$, the pullback map of
differential forms gives rise to a morphism of line bundles $i_f:f^*(\C{K}_\Y)\to \C{K}_\X$.

\end{Emp}

\begin{Emp} \label{E:group}
{\bf The group case}. Let $\G$ be an algebraic group over $F$.

(a) For each $g\in \G$, the left multiplication map $L_g:\G\to \G:x\mapsto gx$ induces an isomorphism
$\Lie \G=T_1(\G)\isom T_g(\G)$ of tangent spaces. These isomorphisms for all $g$ induce a trivialization $\G\times \Lie \G\isom T(\G)$ of the tangent bundle of $\G$.

(b) Consider the one-dimensional vector space $V_\G:=\det(\Lie \G)^*$ over $F$. Then the trivialization from (a) induces an isomorphism $i:\C{O}_\G\otimes_F V_\G\isom \C{K}_\G$ of line bundles.
Explicitly, for every $v\in V_\G$ the differential form $i(v)$ is the unique
left-invariant differential form $\om_\G(v)=\om^l_\G(v)$ on $\G$ such that $\om^l_\G(v)|_{g=1}=v$. Moreover,
$\om^l_\G(v)$ is nonvanishing if $v\neq 0$.

(c) Similarly,  for every $v\in V_\G$ there exists a unique right-invariant
differential form $\om^r_\G(v)$ on $\G$ such that $\om^r_\G(v)|_{g=1}=v$. Explicitly, $\om_{\G}^r(v)=\la_\G\cdot\om_{\G}^l(v)$, where $\la_\G$ was defined in  \re{reg}(a). Observe that if $\G$ is connected reductive, then the character $\la_\G$ is trivial, thus $\om_{\G}^r(v)=\om_{\G}^l(v)$.
\end{Emp}

\begin{Emp} \label{E:canbun}
{\bf The canonical bundle on $\wt{\G}_\H$.}

(a) For every $[g]\in \G/\H$ we have natural identifications $T_{[g]}(\G/\H)\isom\Lie \G/\Lie \H_g$, and
$\pr_{\G/\H}^{-1}([g])\isom \H_g:[g,h]\mapsto ghg^{-1}$ (see \re{grspr}(c)). Thus, for every $[g,h]\in \wt{\G}_\H$, we have an exact sequence
\[
0\to T_{ghg^{-1}}(\H_g)\to T_{[g,h]}(\wt{\G}_\H)\to \Lie \G/\Lie \H_g\to 0.
\]

(b) Using the identification $\Lie \H_g\isom T_{ghg^{-1}}(\H_g)$, induced by $L_{ghg^{-1}}$  as in \re{group}(a), we have a natural isomorphism
\[
\det T_{[g,h]}(\wt{\G}_\H)\cong \det\Lie \H_g\otimes\det(\Lie \G/\Lie \H_g)\cong\det\Lie \G.
\]

(c) By (b), we have a natural isomorphism $i:\C{O}_{\wt{\G}_\H}\otimes_F V_\G\isom \C{K}_{\wt{\G}_\H}$. In particular,
every $0\neq v\in V_\G$ defines a nonvanishing differential form $\om_{\wt{\G}_\H}(v):=i(v)$ on $\wt{\G}_\H$.
%Moreover, $\om_{\wt{\G}_\H}(v)$ has no zeros, if $v\neq 0$.

(d) Let $a_{\H,\G}:\wt{\G}_\H\to\G$ be as in \re{grspr}(c). Using (c) and the isomorphism $\C{K}_\G\isom \C{O}_\G\otimes_F V_\G$ from \re{group}(b), we get a natural isomorphism
\[
\wt{i}:a_{\H,\G}^*(\C{K}_\G)\isom a_{\H,\G}^*(\C{O}_\G\otimes_F V_\G)=\C{O}_{\wt{\G}_\H}\otimes_F V_\G\isom\C{K}_{\wt{\G}_\H}.
\]
\end{Emp}

\begin{Emp} \label{E:right}
{\bf Remarks.}
(a) Recall that the isomorphism $\Lie \H_g\isom T_{ghg^{-1}}(\H_g)$ was constructed in \re{canbun}(b) using the left multiplication.
Therefore the differential form  $\om_{\wt{\G}_\H}(v)$ from \re{canbun}(c) is left-invariant, that is, the restriction of
$\om_{\wt{\G}_\H}(v)$ to each $\pr_{\G/\H}^{-1}([g])\cong \H_g$ is left-invariant.

(b) Instead, we could construct an isomorphism $\Lie \H_g\isom T_{ghg^{-1}}(\H_g)$, using the right multiplication.
As a result,  for every $v\in V_\G$ we would get a right-invariant differential form $\om^r_{\wt{\G}_\H}(v)$ on $\wt{\G}_\H$
such that $\om^r_{\wt{\G}_{\H}}(v)=\la_{\wt{\G}_\H}\cdot\om_{\wt{\G}_\H}(v)$ (compare \re{simple}(a) and \re{group}(c)).

(c) Each differential form $\om_{\wt{\G}_\H}(v)$ is $\G$-invariant. Indeed, it suffices to show that for every $g\in\G$ and $h\in\H$ the
following diagram is commutative:
\[
\begin{CD}
\Lie \H @>L_h>> T_{h}(\H)\\
@V\Ad g VV @VV\Ad g V\\
\Lie \H_g @>L_{ghg^{-1}}>> T_{ghg^{-1}}(\H_g).
\end{CD}
\]
But this follows from the identity $g(hx)g^{-1}=(ghg^{-1})(gxg^{-1})$.
\end{Emp}

%\begin{Emp}
%{\bf The simply connected case.}
%Assume that $\G$ is semisimple and simply connected. Then the homomorphism $\la_{\P}$ has a unique
%square-root $|\cdot|^{1/2}_{\P}:\P\to \B{G}_m$. Indeed, let $U_P\subset P$ be the unipotent radical, and
%$M:=P/U_P$ be the reductive quotient of $\P$, then $|\cdot|_P$ corresponds to the character $2\rho_G-2\rho_{M}$, thus $|\cdot|_P$ corresponds to $\rho_G-\rho_{M}$, which is indeed a character of $\P$, since
%In particular, in this case we can consider  another differential form
%$\om^m_{\wt{G}}(v):=\Dt^{1/2}_P\om_{\wt{G}}^r$ on $\wt{G}$, where $(\cdot)^m$ stays for ``middle".
%Notice that $\om^m_{\wt{G}}(v)$ can be characterized as the unique extension of $\om(v)$, invariant under the map $g\mapsto g^{-1}$.
%\end{Emp}

\begin{Lem} \label{L:meas}
The differential form $a_{\H,\G}^*(\om_\G(v))$ on $\wt{\G}_\H$ is equal to $\Dt_{\wt{\G}_\H}\cdot \om_{\wt{\G}_\H}(v)$, where $\Dt_{\wt{\G}_\H}$ was defined in
\re{simple}(a).
\end{Lem}

\begin{proof}
Identifying all fibers of $\om_\G(v)$ and $\om_{\wt{\G}_\H}(v)$ with $v$ as in \re{group}(b) and \re{canbun}(b) respectively, we have to show that
the Jacobian of the map $a_{\H,\G}:\wt{\G}_\H\to \G$ at $[g,h]$ is $\Dt_{\H,\G}(h)$.  By $\G$-equivariance, we can assume that $g=1$.

Using  \re{canbun}(a), we have two exact sequences of tangent spaces:
\[
0\to \Lie \H\to T_{[1,h]}(\wt{\G}_\H)\to \Lie \G/\Lie \H\to 0
\]
\[
0\to \Lie \H\to T_{h}(\G)=\Lie \G\to \Lie \G/\Lie \H\to 0.
\]
It suffices to show that the differential ${da_{\H,\G}|}_{[1,h]}:T_{[1,h]}(\wt{\G}_\H)\to T_{h}(\G)$ induces the identity on $\Lie \H$ and
the map $\Ad h^{-1}-\Id$ on $\Lie \G/\Lie \H$.

For the first assertion, notice that the restriction of $a_{\H,\G}$ to $\pr_{\G/\H}^{-1}([1])$ is
the inclusion $\H\hra \G$. For the second, notice that the endomorphism of $\Lie \G/\Lie \H$ induced by ${da_{\H,\G}|}_{[1,h]}$, is induced by the differential of the map $\G\to \G:g\mapsto ghg^{-1}$ at $g=1$.
Since $\om_\G(v)$ is left-invariant, the last differential coincides with the differential of the map $\G\to \G:g\mapsto (h^{-1}gh)g^{-1}$, which equals $\Ad h^{-1}-\Id$.
\end{proof}

\begin{Cor} \label{C:etale}
The open subvariety $\wt{\G}_\H^{\reg}\subset \wt{\G}_\H$ is the largest open subvariety, where the map $a_{\H,\G}:\wt{\G}_\H\to \G$ is \'etale.
\end{Cor}

\begin{Emp} \label{E:parab}
{\bf The parabolic case.}  (a) Let $\G$ be connected reductive, $\P\subset \G$ a parabolic subgroup, $\M\subset \P$ a Levi subgroup, and
$\U\subset\P$ the unipotent radical.

(b) Let $\P^-$ be the opposite parabolic of $\P$, and let $\U^-\subset\P^-$ be the unipotent radical.
Then the multiplication map $m:\U^-\times \P\to \G$ is an open embedding. Hence
the maps $\U^-\to \G/\P:u\mapsto [u]$ and $j:\U^-\times \P\to \wt{\G}_\P:(u,x)\mapsto [u,x]$ are open embeddings. In addition, the differential $dm|_{(1,1)}:\Lie \U^-\oplus\Lie \P\to\Lie \G$ is an isomorphism. Therefore it induces an isomorphism $V_{\U^-\times \P}\isom V_\G$.
\end{Emp}

\begin{Cl} \label{C:parab}
For every $m\in \M$, we have the identity
\[\Dt_{\P,\G}(m)^2=(-1)^{\dim \U}\Dt_{\M,\G}(m)\cdot\la_\P(m).\]
\end{Cl}

\begin{proof}
Since both sides of the equality are regular functions on $\M$, we may assume that $m$ is regular semisimple, and thus lies in a maximal torus $\T\subset \M$. Now we can decompose $\Lie \G/\Lie \M$ as a sum of $\T$-eigenspaces. Now the equality follows from the identity $(t-1)^2=(-1)(t-1)(t^{-1}-1)t$.
\end{proof}

\begin{Lem} \label{L:open}
In the notation of \re{parab}(b),  we have equalities
$j^*(\om_{\wt{\G}_\P}^l(v))=\om^l_{\U^-\times \P}(v)$ and $m^*(\om_{\G}(v))=\om^r_{\U^-\times \P}(v)$ for every $v\in V_\G$.
\end{Lem}

\begin{proof}
Since the fiber of each $j^*(\om_{\wt{\G}_\P}^l(v))$ and $\om^l_{\U^-\times \P}(v)$ at $1$ is $v$, for the first equality it suffices to show that $j^*(\om_{\wt{\G}_\P}^l(v))$ is left $(\U^-\times \P)$-invariant. For the  $\U^-$-invariance, notice that $j$ is $\U^-$-equivariant, and $\om_{\wt{\G}_\P}^l(v)$ is $\G$-invariant (see \re{right}(c)). For the $\P$-invariance, notice that $\om_{\wt{\G}_\P}^l(v)|_{[1]\times \P}$ is left $\P$-invariant (see \re{right}(a)).

For the second equality, one has to show that $m^*(\om_{\G}(v))$ is right $(\U^-\times \P)$-invariant.
Since $\om_{\G}(v)$ is $\G\times \G$-invariant, the pullback $m^*(\om_{\G}(v))$ is right $\P$-invariant and left $\U^-$-invariant.
Since the group $\U^-$ is unipotent, we have $\la_{\U^-}=1$, thus the differential form $m^*(\om_{\G}(v))$ is also right $\U^-$-invariant.
\end{proof}

\section{Smooth measures, restriction, and induction}

\noindent From now on, let $F$ be a local non-archimedean field,  and let $|\cdot|:F\m\to\B{R}\m$ be the norm map.
For every compact analytic subgroup $K$ over $F$, we denote by $\dt_K$ the Haar measure on $K$ with total measure $1$.

\begin{Emp}
{\bf Smooth measures.} Let $X$ be a smooth analytic variety over $F$.

 (a) Let  $C^{\infty}(X)$ (resp. $C_c^{\infty}(X)$) be the space of smooth (complex-valued) functions (resp.\,with compact support), and let
 $\C{M}(X)$ be the dual space of $C_c^{\infty}(X)$. Every non-vanishing ($F$-valued) analytic function $f$ on $X$ induces a smooth function $|f|\in C^{\infty}(X)$, defined by $|f|(x):=|f(x)|$ for every $x\in X$.

(b) We say that a measure $\chi\in \C{M}(X)$ is {\em smooth}, and write $\chi\in\C{M}^{\infty}(X)$,
if for every $x\in X$, there exists an open neighborhood $U\subset X$ of $x$ and an analytic isomorphism
$\phi:\C{O}_F^n\isom U$ such that $\phi^*(\chi|_{U})\in\C{M}(\C{O}_F^n)$ is a multiple of a Haar measure on $\C{O}_F^n$.

(c) By a construction of Weil \cite{We}, every non-vanishing top degree differential form $\om$ on $X$ defines a smooth measure $|\om|\in\C{M}^{\infty}(X)$.

Namely, for every open analytic embedding $\phi:\C{O}_F^n\isom U\subset X$, the differential form $\phi^*(\om)$  equals $fdx_1\wedge\ldots\wedge dx_n$ for some non-vanishing analytic function $f$ on $\C{O}_F^n$. Then $|f|\in C_c^{\infty}(\C{O}_F^n)$, and $|\om|$ is characterized by the condition that the pullback $\phi^*(|\om|)$ equals $|f|\dt_{\C{O}_F^n}\in \C{M}^{\infty}(\C{O}_F^n)$.
\end{Emp}

Now we give an equivalent (more geometric) definition of smooth measures.

\begin{Emp} \label{E:global}
{\bf An alternative description.}
Let $\C{K}^{-1}_X$ be the line bundle dual to $\C{K}_X$, and let $\Si_X\to X$ be the $F\m$-torsor corresponding to $\C{K}^{-1}_X$. Explicitly, $\Si_X$
is a space of pairs $(x,a)$, where $x\in X$ and $a$ is a non-zero element of the fiber $\C{K}^{-1}_X|_x$ of $\C{K}^{-1}_X$ at $x$.
\end{Emp}

\begin{Cl} \label{C:global}
The space $\C{M}^{\infty}(X)$ is canonically identified with the space $C^{\infty}(\Si_X,|\cdot|)$ of smooth functions $f:\Si_X\to\B{C}$ satisfying $f(bx)=|b|f(x)$ for all $b\in F\m$.
\end{Cl}

\begin{proof}
Since both $\C{M}^{\infty}(X)$ and $C^{\infty}(\Si_X,|\cdot|)$ are defined as global sections of certain sheaves on $X$, we can construct an isomorphism  $\C{M}^{\infty}(X)\isom C^{\infty}(\Si_X,|\cdot|)$ locally on $X$.
Thus we can assume there exists a non-vanishing differential form $\om\in \Gm(X,\C{K}_X)$.

Note that the tensor product $s\mapsto s\otimes\om$ defines an isomorphism of line bundles $\C{K}^{-1}_X\isom\C{O}_X$. Hence it induces an isomorphism $i_{\om}:\Si_X\isom F\m\times X$ of the corresponding $F\m$-torsors over $X$. For every
$f\in C^{\infty}(X)$, we denote by $\wt{f}\in C^{\infty}(F\m\times X,|\cdot|)$ the function $(a,x)\mapsto |a|f(x)$. The map $f|\om|\mapsto i_{\om}^*(\wt{f})$ defines an isomorphism $\C{M}^{\infty}(X)\isom C^{\infty}(\Si_X,|\cdot|)$, which is independent of $\om$.
\end{proof}

\begin{Emp} \label{E:smmeas}
{\bf Smooth measures with compact support.}
We denote by
 $\C{M}_c(X)$ (resp. $\C{H}(X)=\C{M}^{\infty}_c(X)$) the space of measures (resp.\,smooth measures) on $X$ with compact support. %Notice that $\C{H}(X)$ is the usual Hecke algebra when $X$ is a group.
Moreover, if a group  $G$ acts on $X$, then $G$ acts on the space $\C{H}(X)$, and we denote by $\C{H}(X)^G$ and $\C{H}(X)_G$
the spaces of $G$-invariants and $G$-coinvariants, respectively.
\end{Emp}

\begin{Emp} \label{E:pullback}
{\bf Pullback of smooth measures.} (a) Assume that we are given a morphism $f:X\to Y$ of smooth analytic varieties and an  isomorphism $i:f^*(\C{K}_Y)\isom\C{K}_X$ of line bundles. Then $i$ induces an isomorphism of line bundles $\C{K}^{-1}_X\isom  f^*(\C{K}^{-1}_Y)$, hence a morphism of $F\m$-torsors $(f,i):\Si_X\isom X\times_Y \Si_Y\to \Si_Y$.

By \rcl{global},
$(f,i)$ induces a pullback map $(f,i)^*:\C{M}^{\infty}(Y)\to \C{M}^{\infty}(X)$. If, in addition, $f$ is proper, then $(f,i)^*$ induces a pullback $(f,i)^*:\C{H}(Y)\to \C{H}(X)$.

(b) Note that if $f:X\to Y$ is a local isomorphism, then the morphism of line bundles $i_f:f^*(\C{K}_Y)\to\C{K}_X$ from \re{Om}(b) is an isomorphism.
Therefore, by (a), $f$ gives rise to a  pullback map $f^*=(f,i_f)^*:\C{M}^{\infty}(Y)\to \C{M}^{\infty}(X)$.
Moreover, $f$ induces a pullback map $f^*:\C{H}(Y)\to \C{H}(X)$, if in addition $f$ is proper.

In particular, if $X\hra Y$ is an open (resp. open and closed) embedding, then we have a  restriction map $\res:\C{M}^{\infty}(Y)\to \C{M}^{\infty}(X)$
(resp. $\res:\C{H}(Y)\to \C{H}(X)$).
\end{Emp}

The following simple lemma is basic for what follows.

\begin{Lem} \label{L:push}
(a) Let $f:X\to Y$ be a smooth map between smooth analytic varieties.
Then the pushforward map $f_!:\C{M}_c(X)\to \C{M}_c(Y)$ satisfies
$f_!(\C{H}(X))\subset \C{H}(Y)$, that is, $f_!(h)\in \C{H}(Y)$ for every $h\in \C{H}(X)$.

(b) Let $G$ be an analytic group, and let $f:X\to Y$ be a principal $G$-bundle. Then the map $f_!$ induces an isomorphism
$\C{H}(X)_G\isom \C{H}(Y)$.
\end{Lem}

\begin{proof}
(a) The question is local in $X$ and $Y$, so we may assume that
$X=\C{O}_F^{n+m}$, $Y=\C{O}_F^{n}$, $f$ is the projection, and $h=\dt_X$.
In this case, $f_!(h)=\dt_Y$.

(b) The assertion is local in $Y$, and $f$ is locally trivial, so we may assume that $X=G\times Y$.
In this case, for every compact open subgroup $K\subset G$, the map $\C{H}(Y)\to \C{H}(X)_G$ defined by $h\mapsto [h\pp\dt_K]$
is the inverse of $f_!:\C{H}(X)_G\isom \C{H}(Y)$.
\end{proof}

\begin{Emp} \label{E:ind}
{\bf The induced space.}
(a) Let $G$ be an analytic group, $H\subset G$ a closed analytic subgroup, and let $H$ act on a smooth analytic variety $X$. Then the product $G\times X$ is equipped with an action of $G\times H$ defined by $(g,h)(g',x):=(gg'h^{-1},h(x))$, and the quotient $\Ind_H^G(X):=G\overset{H}{\times}X$ is smooth and equipped with an action of $G$.

(b) Consider the diagram $\Ind_H^G(X)\overset{p_1}{\lla}G\times X \overset{p_2}{\lra}X$, where $p_1$ and $p_2$ are the natural projections. Since $p_1$ is a $G$-equivariant $H$-bundle, while $p_2$ is an $H$-equivariant $G$-bundle, \rl{push}(b)
implies that we have a natural isomorphism
\[
\varphi_H^G:=(p_2)_!\circ (p_1)_!^{-1}:\C{H}(\Ind_H^G(X))_G\bisom \C{H}(G\times X)_{G\times H}\isom \C{H}(X)_H.
\]

(c) By construction, the isomorphism $\varphi_H^G$ from (b) can be characterized as the unique map
$\varphi_H^G:\C{H}(\Ind_H^G(X))_G\to \C{H}(X)_H$ such that the composition $\varphi_H^G\circ (p_1)_!:\C{H}(G\times X)_{G\times H}\to\C{H}(X)_H$
is equal to $(p_2)_!$

(d) Assume that $H$ acts on $X$ trivially. Then $\Ind_H^G(X)=(G/H)\times X$, and the natural projection
$\pr:(G/H)\times X\to X$ satisfies  $\pr\circ p_1=p_2$. Thus, by (c),
the map $\varphi_H^G$ obtained from $\pr_!:\C{H}((G/H)\times X)\to \C{H}(X)$.

(e) Assume that $H$ is a retract of $G$, and let $p:G\to H$ be a homomorphism such that $p|_H=\Id_H$. Then $p$ induces a map
$p:\Ind_H^G(X)\to \Ind_H^H(X)=X$ such that $p\circ p_1=p_2$. Thus, by (c), the map
$\varphi_H^G$ is obtained from $p_!:\C{H}(\Ind_H^G(X))\to \C{H}(X)$.
\end{Emp}

\begin{Emp} \label{E:alcase}
{\bf Notation.} For every (smooth) algebraic variety $\X$ over $F$, we denote by $X$ the corresponding (smooth) analytic variety $\X(F)$. In particular, we have $G=\G(F)$, $\wt{G}_H=\wt{\G}_\H(F)$, etc. We also denote by $a_{H,G}:\wt{G}_H\to G$, $\Dt_{H,G}:H\to F$, $\la_G:G\to F\m$, etc., the maps induced by $a_{\H,\G}$, $\Dt_{\H,\G}$ and $\la_\G$, respectively.
\end{Emp}

\begin{Emp} \label{E:main}
{\bf The restriction map.} In the case of \re{parab}, set $\wt{G}:=\wt{G}_P$ and $a:=a_{P,G}$.

(a) In \re{canbun}(d) we constructed an $\Ad G$-equivariant isomorphism $\wt{i}:a^*(\C{K}_G)\isom \C{K}_{\wt{G}}$, which by \re{pullback}(a) defines
a pullback map $\wt{a}^*=(a,\wt{i})^*:\C{M}^{\infty}(G)\to\C{M}^{\infty}(\wt{G})$.

Explicitly, for every $v\in V_G$ and $f\in C^{\infty}(G)$, the map $\wt{a}^*$ is given by the formula
$\wt{a}^*(f|\om_G(v)|):=a^*(f)|\om^l_{\wt{G}}(v)|$. Moreover, since $a:\wt{G}\to G$ is proper,
the map  $\wt{a}^*$ induces a $G$-equivariant map $\wt{a}^*:\C{H}(G)\to\C{H}(\wt{G})$.

(b) Since $(\G/\P)(F)$ equals $\G(F)/\P(F)=G/P$ (see \cite[Prop 20.5]{Bo}),
the set $\wt{G}=\wt{\G}(F)$ equals $\Ind_P^G(P)$. Hence, by \re{ind}(b), we have a natural isomorphism
\[
\varphi_P^G:\C{H}(\wt{G})_{G}=\C{H}(\Ind_{P}^{G}(P))_{G}\isom \C{H}(P)_{P}.
\]
We denote by $\Res_P^G:\C{H}(G)_{G}\to \C{H}(P)_{P}$ the composition  map $\varphi_P^G\circ \wt{a}^*$.

(c) Let $p:\P\to \M$ be the projection (see \rl{eqrk}). Since $p$ is smooth, the induced map $p:P\to M$ is smooth.
Thus $p$ induces a map $p_!:\C{H}(P)_{P}\to \C{H}(M)_{M}$ (see \rl{push}(a)), and we denote by $\Res_{P;M}^G:\C{H}(G)_{G}\to \C{H}(M)_{M}$ the composition $p_!\circ\Res_P^G$.
\end{Emp}

\begin{Emp} \label{E:gen}
{\bf Generalized functions and induction.}
(a) Let $X$ be a smooth analytic variety over $F$, and let $\wh{C}(X):=\Hom_{\B{C}}(\C{H}(X),\B{C})$ be the space generalized functions on $X$. Then the space $\wh{C}(X)$ is equipped with a {\em weak topology}, which is the coarsest topology such that for every $h\in\C{H}(X)$ and $a\in\B{C}$,
the set $\{\la\in \wh{C}(X)\,|\,\la(h)\neq a\}$ is open.

(b) For a group $G$  acting on $X$, we denote by $\wh{C}^G(X)\subset\wh{C}(X)$ the subspace of invariant generalized functions with the induced topology. Equivalently, $\wh{C}^G(X)$ is the linear dual space of $\C{H}(X)_G$ equipped with the weak topology.

(c) By \rl{push}(a), every smooth map $f:X\to Y$ induces a continuous map $f^*:\wh{C}(Y)\to \wh{C}(X)$ dual to $f_!$. If, in addition, $f$ is $G$-equivariant, then $f^*$ induces a continuous map $\wh{C}^G(Y)\to \wh{C}^G(X)$.

(d) In the case of \re{ind}, we have a linear homeomorphism
$\wh{C}^H(X)\isom\wh{C}^G(\Ind_H^G(X))$, dual to $\varphi_H^G$.

(e) In the case of \re{main}, we denote by $\Ind_{P;M}^G:\wh{C}^{M}(M)\to \wh{C}^{G}(G)$ (resp.
$\Ind_P^G:\wh{C}^{P}(P)\to \wh{C}^{G}(G)$), the map dual to $\Res_{P;M}^G$ (resp. $\Res_{P}^G$).

(f) Let $\H$ be an algebraic group over $F$. Then to every admissible representation $\pi$ of $H=\H(F)$  we can associate its character $\chi_{\pi}\in\wh{C}^{H}(H)$.
\end{Emp}

\section{Relation to characters of induced representations}

\noindent Assume that we are in the case of \re{main}.

%Since the subset $\bO _M U_0\subset P_0$ is Ad-invariant we can consider the quotient $Y$ of $G\times bO _M U_0/P_0$ where $P_0$ acts  by
%left shifts on $G$ and by conjugation on $\bO _M U_0$. Let $\ti h:G\times \bO _M U_0\to G$ by given by $\ti h(g,a)=g^{-1}ag$.
%It is easy to see that $\ti h$ induces an isomorphism $Y\to \bO _G$

%Lemma 2 follows immediately from this equality.

\begin{Prop} \label{P:indrep}
Let $\tau$ be an admissible representation of $P$, and let
$\pi=\un{\Ind}_{P}^{G}(\tau)$ be the induced representation. Then
we have the equality $\chi_{\pi}=\Ind_P^G(\chi_\tau)$.
\end{Prop}

To prove the result, we will compute both $\chi_{\pi}$ and $\Ind_P^G(\chi_\tau)$ explicitly.

\begin{Emp} \label{E:form}
{\bf Notation.}
(a) Let $K\subset G$ be a compact open subgroup, set $K_P:=K\cap P$, and let  $\mu^K$ (resp. $\mu^{K_{P}}$)
be the left-invariant Haar measure on $G$ (resp.\,$P$) normalized by the condition that $\mu_K(K)=1$ (resp. $\mu^{K_{P}}(K_{P})=1$).

(b) Set $\wt{K}:=K\overset{K_P,\Ad}{\times} P=\Ind_{K_P}^{K}(P)$. Then $\wt{K}\subset\wt{G}$ is an open and closed subset, and we set
\[
\Res_{K_P}^K:\C{H}(G)_{K}\overset{\wt{a}^*}{\lra}\C{H}(\wt{G})_{K}\overset{\res}{\lra}\C{H}(\wt{K})_{K}\overset{\varphi^{K}_{K_P}}{\lra}\C{H}(P)_{K_P},
\]
where $\wt{a}^*$ was defined in \re{main}(a) and $\varphi^{K}_{K_P}$  was defined in \re{ind}(b).

(c) For every $h\in \C{H}(G)^K$, we define $f:=h/\mu^K\in C_c^{\infty}(G)$, $f_P:=f|_{P}\in C_c^{\infty}(P)$ and  $h_{P}:=f_P\mu^{K_{P}}\in \C{H}(P)$.

(d) For every $g\in G$, we set $K_g:=gKg^{-1}$, $h_g:=(\Ad g^{-1})^*(h)\in\C{H}(G)^{K_g}$,
$K_{g,P}:=K_g\cap P$ and $h_{g,P}:=(h_g)_P$ (see (c)).

(e) Fix a set of representatives $A\subset G$ of double classes  $P\bs G/K$. The set $A$ is finite, because $P\bs G$ is compact.
\end{Emp}

\begin{Lem} \label{L:form}
In the notation of \re{form}, we have equalities:

(a)  $\Res_P^G([h])=\sum_{g\in A} \Res_{K_{g,P}}^{K_g}([h_g])$ in $\C{H}(P)_{P}$;

(b) $\Res_{K_{P}}^{K}([h])=[h_P]$ in $\C{H}(P)_{K_P}$;

(c)  $\Res_P^G([h])=\sum_{g\in A} [h_{g,P}]$ in $\C{H}(P)_{P}$.
\end{Lem}

\begin{proof}
(a) By definition,  $\Res_P^G([h])=\varphi_P^G([\wt{a}^*(h)])$.
Notice that the decomposition $G=\sqcup_{g\in A} Kg^{-1}P$ into open and closed subsets induces a
decomposition $\wt{G}=\sqcup_{g} \wt{G}_{g}$, where $\wt{G}_{g}:=Kg^{-1}P\overset{P}{\times}P$.
Therefore we get a decomposition $\wt{a}^*(h)=\sum_{g}\{h\}_{g}$, where $\{h\}_{g}:=\wt{a}^*(h)|_{\wt{G}_{g}}\in \C{H}(\wt{G}_g)\subset \C{H}(\wt{G})$.
Thus it remains to show that for every $g\in G$, we have the equality
$\varphi_P^G([\{h\}_{g}])=\Res_{K_{g,P}}^{K_g}([h_g])$.

Note that $g\wt{G}_{g}=gKg^{-1}P\overset{P}{\times}P=K_g\overset{K_{g,P}}{\times}P=\wt{K_g}$ and $(g^{-1})^*(\{h\}_{g})=\wt{a}^*(h_g)|_{\wt{K_g}}$.
Using the identity
%$\C{H}(\wt{G}_{K_g})$.
$[\{h\}_{g}]=[(g^{-1})^*(\{h\}_{g})]\in\C{H}(\wt{G})_G$, the equality
$\varphi_P^G([\{h\}_{g}])=\Res_{K_{g,P}}^{K_g}([h_g])$ can be rewritten as
$\varphi_P^G([\wt{a}^*(h_g)|_{\wt{K_g}}])=\Res_{K_{g,P}}^{K_g}([h_g])$.

The latter equality follows from the fact that the diagram
\[
\begin{CD}
\C{H}(\wt{G}) @>\varphi_P^G>> \C{H}(P)_{P}\\
@A(1)AA @A(2)AA\\
\C{H}(\wt{K_g}) @>\varphi_{K_{g,P}}^{K_g}>> \C{H}(P)_{K_{g,P}},
\end{CD}
\]
where (1) is the natural inclusion, while (2) is the natural projection, is commutative.

(b) Recall that $\varphi^K_{K_P}$ is a composition $\C{H}(\wt{K})_K\bisom \C{H}(K\times P)_{K\times K_P}\isom  \C{H}(P)_{K_P}$, corresponding to the diagram $\wt{K}\overset{p_1}{\lla}K\times P \overset{p_2}{\lra}P$.
Since $\dt_K\pp h_P\in\C{H}(K\times P)$ satisfies $(p_2)_!(\dt_K\pp h_P)=h_P$, it suffices to show
that $(p_1)_!(\dt_K\pp h_P)=\wt{a}^*(h)|_{\wt{K}}$.

Note that $\wt{a}^*(h)=a^*(f)\wt{a}^*(\mu^K)$, while $\dt_K\pp h_P=(1_K\pp f_P)(\dt_K\times \mu^{K_P})$.
Moreover, since $f$ is $\Ad K$-invariant, we have $p_1^*(a^*(f))=1_K\pp f_P$. Thus it remains to show the equality
$(p_1)_!(\dt_K\pp\mu^{K_P})=\wt{a}^*(\mu^K)|_{\wt{K}}$.

Fix  $0\neq v\in V_G$ and $0\neq v'\in V_P$. Using the identities $|\om_G(v)|=|\om_G(v)|(K)\cdot\mu^K$ and
$|\om^l_P(v')|=|\om^l_P(v')|(K_P)\cdot\mu^{K_P}$, it remains to show the equality
\begin{equation} \label{Eq:push}
(p_1)_!(|\om_K(v)|\pp|\om^l_P(v')|)=|\om^l_P(v')|(K_P)\cdot |\om_{\wt{K}}(v)|,
\end{equation}
where we set $\om_{\wt{K}}(v)=\om_{\wt{G}}(v)|_{\wt{K}}$. Using the notation of \re{parab}(b), we set $K^-:=K\cap U^-$, and consider the open embeddings $j:K^-\times P\hra \wt{K}$, $m:K^-\times K_P\hra K$ and $K^-\times P\hra G$. Since $p_1$ is $K$-equivariant, it remains to show the restriction of the equality \form{push} to $K^-\times P$ under $j$.

 Note that $p_1^{-1}(K^-\times P)=K^-\times K_P\times  P$, while  $j^*(\om_{\wt{K}}(v))=\om^l_{K^-\times P}(v)$ and $m^*(\om_K(v))=\om^r_{K^-\times K_P}(v)$ (by \rl{open}). Thus the assertion follows from the fact that $|\om^r_{K^-\times K_P}(v)|=|\om^l_{K^-\times K_P}(v)|$, because $K^-\times K_P$ is compact.

(c) follows immediately from (a) and (b).
\end{proof}

\begin{Emp}
\begin{proof}[Proof of \rp{indrep}]
We have to show that for every $h\in \C{H}(G)$, we have
$\chi_{\pi}(h)=\chi_\tau(\Res_P^G(h))$. Choose an open compact subgroup $K\subset G$ such that
$h$ is $K\times K$-invariant. Then $h\in\C{H}(G)^K$, so by \rl{form} we have to show that
\begin{equation}
\Tr(h,\pi)=\sum_{g\in A}\Tr(h_{g,P},\tau).
\end{equation}
Though the result is well-known and is an immediate generalization of the corresponding
result for finite groups, we sketch the argument for completeness.

Let $W$ (resp. $V$) be the space of $\tau$ (resp. $\pi$).
Then $V^K$ is the space of functions $f:G\to W$ satisfying $f(xyk)=\tau(x)(f(y))$ for all
$x\in P$, $y\in G$ and $k\in K$.
For every $g\in G$, we denote by $V^K_{g}\subset V^K$ the subspace of functions $f:G\to W$ from $V^K$,
supported on $PgK$. Then $V^K=\oplus_{g\in A}V^K_g$.

For every $g\in G$, consider the endomorphism $h_{\{g\}}:V^K_{g}\hra V^K\overset{h}{\to}V^K\surj V^K_{g}$, induced by $h$. Then
$\Tr(h,\pi)=\sum_{g\in A}\Tr(h_{\{g\}})$, so it suffices to show  the equality
$\Tr(h_{\{g\}})=\Tr(h_{g,P},W)$ for every $g\in G$.

Since $gK=K_g g$, the map
$f\mapsto f(g)$ induces a linear isomorphism $V^K_{g}\isom W^{K_{g,P}}$.
It remains to show that this isomorphism identifies
$h_{\{g\}}\in \End V^K_{g}$ with $\tau(h_{g,P})\in \End W^{K_{g,P}}$, that is, for every
$f\in V^K_{g}$, we have the equality $(h(f))(g)=\tau(h_{g,P})(f(g))$.

We claim that the latter equality holds for every right $K$-invariant $h\in \C{H}(G)$.
Indeed, we may assume that $h=\dt_{xK}$ for some $x\in G$. Set $x_g:=gxg^{-1}$. Then $h_g=\dt_{x_gK_g}$, and  $(h(f))(g)=f(gx)=f(x_g g)$. Assume first that $x_g\in PK_g$. Then $x_gK_g=x'K_g$ for some $x'\in P$.
In this case, we have $h_{g,P}=\dt_{x'K_{g,P}}$, and
\[
h(f)(g)=f(x_g g)=f(x' g)=\tau(x')(f(g))=\tau(h_{g,P})(f(g)).
\]
Finally, if $x_g\notin PK_g$, then $h(f)(g)=f(x_g g)=0$, and $h_{g,P}=0$.
\end{proof}
\end{Emp}

\begin{Emp}
{\bf Parabolic induction.} Let $\rho$ be an admissible representation of $M$. Recall that a non-normalized parabolic induction
$\pi=\un{\Ind}_{P;M}^G(\rho)$ is the induced representation $\un{\Ind}_{P}^G(\tau)$, where
$\tau\in\Rep(P)$ is the inflation of $\rho$.
\end{Emp}

\begin{Cor} \label{C:ind}
We have the equality of characters $\chi_{\pi}=\Ind_{P;M}^G(\chi_\rho)$.
\end{Cor}
\begin{proof}
Since the character of the inflation $\chi _{\tau}\in \wh{C}^{P}(P)$ equals $p^*(\chi_\rho)$, the assertion follows from \rp{indrep}.
\end{proof}

\section{Restriction to an equal rank subgroup}

\begin{Emp} \label{E:relcomp}
{\bf Smooth measures with relatively compact support.}

(a) Let $f:X\to Z$ be a morphism of analytic varieties over $F$, where $X$ is smooth. We denote by
$\C{H}(X/Z)\subset \C{M}^{\infty}(X)$ the subspace consisting of measures, whose support is proper
over $Z$. In particular, $\C{H}(X/X)=\C{M}^{\infty}(X)$. Notice that $\C{H}(X)\subset \C{H}(X/Z)$, if $Z$ is Hausdorff, and
$\C{H}(X)=\C{H}(X/Z)$, if $Z$ is compact.

(b) Every smooth morphism $f:X\to Y$ over $Z$ induces a canonical push-forward map $f_!:\C{H}(X/Z)\to \C{H}(Y/Z)$. Indeed,
we can construct the map $f_!$ locally on $Z$, thus may assume that $Z$ is compact. Now the assertion follows from
\rl{push}.

(c) Assume that we are in the case of \re{ind}, and that $H$ acts on $X$ over $Z$, that is, the map $f:X\to Z$ is $H$-equivariant with respect to the
trivial action of $H$ on $Z$. Then the arguments of (b) and \re{ind} imply that we have a natural isomorphism $\varphi_H^G:\C{H}(\Ind_H^G(X)/Z)_G\isom \C{H}(X/Z)_H$.

(d) Assume that we are given a commutative diagram of analytic spaces over $F$
\[
\begin{CD}
X @>>>Y\\
@VbVV @VVV\\
X'@>>> Y',
\end{CD}
\]
such that $X$ and $X'$ are smooth, $b$ is a local isomorphism, and the induced map $X\to X'\times_{Y'} Y$ is proper.
Then the pullback $b^*:\C{M}^{\infty}(X')\to\C{M}^{\infty}(X)$ (see \re{pullback}(b)) satisfies
$b^*(\C{H}(X'/Y'))\subset\C{H}(X/Y)$.

(e) An important particular case of (d) is when the diagram is Cartesian and $Y\to Y'$ is an open embedding.
In this case,  $b:X\to X'$ is an open embedding as well, and we denote $b^*$ by $\res$ and call it
{\em the restriction map}.
\end{Emp}

\begin{Emp} \label{E:alg}
{\bf The algebraic case.}
%(a) Let $f:X\to Z$ be a morphism of algebraic varieties over $F$, where $X$ is smooth. To shorten the notation, we will
%write $\C{H}(X)$ (resp. $\C{H}(X/Z)$) instead of $\C{H}(X)$ (resp. $\C{H}(X/Z)$). Moreover, if an algebraic group
%$G$ acts  $X$ (resp. over $Z$), we will write  $\C{H}(X)_G$ (resp. $\C{H}(X/Z)_G$ instead of $\C{H}(X)_{G}$ (resp. $\C{H}(X/Z)_{G}$)).
%
(a) Let $\H\subset \G$ be a closed subgroup. Since the projection map $\G\to \G/\H$ is smooth,
 every $G$-orbit in $(\G/\H)(F)$ is open. Therefore $G/H=\G(F)/\H(F)$ is an open and closed subset of
$(\G/\H)(F)$.

(b) Let $\H$ act on a smooth algebraic variety $\X$. Then, by (a), the induced space
$\Ind^G_H(X)$ is an open and closed subset of
$(\Ind^\G_\H(\X))(F)$. Hence we have a natural restriction map
$\res:\C{H}((\Ind_\H^\G(\X))(F))_G\to \C{H}(\Ind_{H}^{G}(X))_G$, and we denote  the composition $\varphi_{H}^{G}\circ\res:
\C{H}((\Ind_\H^\G(\X))(F))_G\to \C{H}(X)_H$ simply by $\varphi_{H}^{G}$.

(c) Assume that  $\H$ acts on $\X$ over $\Z$ (compare \re{relcomp}(c)).
Then, generalizing (b),  we have a  map $\varphi^G_H: \C{H}((\Ind_\H^\G(\X))(F)/Z)_G\to \C{H}(X/Z)_H$.
\end{Emp}

\begin{Emp} \label{E:notrestr}
{\bf Notation.} Let $\K\subset \H$ be two closed subgroups of $\G$.

(a) We set $\C{H}(H)^{\reg/G}:=\C{H}(H^{\reg/G}/c_H^{\reg/G})$,
where $H^{\reg/G}=\H^{\reg/\G}(F)$ and $c_H^{\reg/G}=\c_\H^{\reg/\G}(F)$ (see \re{reg}(b)).
Similarly, we set $\C{H}(\wt{G}_H)^{\reg}:=\C{H}(\wt{G}_H^{\reg}/c_H^{\reg/G})$ and
$\C{H}(\wt{H}_K)^{\reg/G}:=\C{H}(\wt{H}_K^{\reg/G}/c_K^{\reg/G})$
(see \re{simple}(e)).

(b) By a combination of \re{relcomp}(a) and \re{relcomp}(e), we have a restriction map
\[\res:\C{H}(H)\hra \C{H}(H/c_H)\to \C{H}(H)^{\reg/G}.\]

(c) Recall that the map $a_{\H,\G}:\wt{\G}^{\reg}_\H\to \G$  is
$\G$-equivariant and \'etale (see \rco{etale}); thus the corresponding
map $a_{H,G}:\wt{G}^{\reg}_H\to G$ is a local isomorphism. Hence we have a pullback map
$a_{H,G}^*:\C{M}^{\infty}(G)\to\C{M}^{\infty}(\wt{G}^{\reg}_H)$  (see \re{pullback}(b)).
\end{Emp}

\begin{Emp} \label{E:restr}
{\bf The restriction map.} Assume that we are in the case of \re{eqrk}.

(a) Recall that the morphism of algebraic varieties $\iota_{\H,\G}:\wt{\G}^{\reg}_\H\to \G\times_{\c_\G}\c^{\reg/\G}_\H$ is finite (by \rco{eqrank}). Therefore the induced morphism of analytic varieties $\iota_{H,G}:\wt{G}^{\reg}_H\to G\times_{c_G}c^{\reg/G}_H$ is proper. Thus, by \re{notrestr}(c) and \re{relcomp}(d), we have a pullback map $a_{H,G}^*:\C{H}(G)_G\to\C{H}(G/c_G)_G\to \C{H}(\wt{G}_H)_G^{\reg}$.

(b) We denote by $R^G_H:\C{H}(G)_G\to \C{H}(\wt{G}_H)_G^{\reg}\to \C{H}(H)^{\reg/G}_H$, the composition of
the map $a_{H,G}^*$ from (a) and the map $\varphi_H^G$ from \re{alg}(c).
\end{Emp}

\begin{Emp} \label{E:setup}
{\bf Set-up.} Assume that we are in the case of \re{main}, and that $\H$ is a connected equal rank subgroup of $\M$,
hence also of $\G$.

(a) Then $\c_\H^{\reg/\G}\subset \c_\H^{\reg/\M}$ is an open subscheme (see \re{reg}(c)), and thus we have defined a restriction map
$\res:\C{H}(H)^{\reg/M}_H\to \C{H}(H)^{\reg/G}_H$ (see \re{relcomp}(e)).

(b) Using \re{reg}(c) again, we conclude that $\H^{\reg/\G}\subset \P^{\reg/\G}$. Thus
the restriction of $\Dt_{\P,\G}$ to $\H^{\reg/\G}$ is non-vanishing. Hence $\Dt_{\P,\G}$ gives rise to a smooth function
$|\Dt_{P,G}|\in\C{M}^{\infty}( H^{\reg/G})$.
\end{Emp}

\begin{Lem} \label{L:res}
In the case of \re{setup},  the following diagram is commutative:
\begin{equation} \label{Eq:genlevi}
\begin{CD}
\C{H}(G)_{G} @>R_H^G>> \C{H}(H)^{\reg/G}_H\\
  @V\Res_{P;M}^GVV   @A|\Dt_{P,G}|\cdot {\res}AA\\
\C{H}(M)_{M}  @>R_H^M>> \C{H}(H)^{\reg/M}_H.
\end{CD}
\end{equation}
\end{Lem}

\begin{proof} The assertion follows from a rather straightforward diagram chase. Namely, using the inclusion $H^{\reg/G}\subset M^{\reg/G}$, we observe that diagram \form{genlevi} decomposes as
\begin{equation} \label{Eq:gl1}
\begin{CD}
\C{H}(G)_{G} @>R_M^G>> \C{H}(M)^{\reg/G}_M @>R_H^M>> \C{H}(H)^{\reg/G}_H\\
  @V\Res_{P;M}^GVV   @A|\Dt_{P,G}|\cdot {\res}AA       @A|\Dt_{P,G}|\cdot {\res}AA\\
\C{H}(M)_{M}  @= \C{H}(M)_M  @>R_H^M>> \C{H}(H)^{\reg/M}_H.
\end{CD}
\end{equation}
Since the right-inner square of \form{gl1} is commutative by functoriality, it remains to show the commutativity of the left-inner square.

Observe that the diagram
\begin{equation} \label{Eq:gl2}
\begin{CD}
\C{H}(P)_{P} @>\res >>\C{H}(P)^{\reg/G}_P @>a_{M,P}^*>> \C{H}(\wt{P}_M)^{\reg/G}_P\\
  @Vp_!VV   @Vp_!VV       @V\varphi_M^P VV\\
\C{H}(M)_{M}  @>\res >>\C{H}(M)^{\reg/G}_M  @=\C{H}(M)^{\reg/G}_M
\end{CD}
\end{equation}
is commutative. Indeed, the left-inner square of \form{gl2} is commutative by \rl{eqrk}(d), and the right-inner square is commutative, because $\varphi_M^P=p_!$ (see \re{ind}(e)) and $a_{P,M}^*$ is an isomorphism (by \rl{eqrk}(e)).

Therefore the left-inner square of \form{gl1} decomposes as
\begin{equation} \label{Eq:gl3}
\begin{CD}
 \C{H}(G)_{G} @=\C{H}(G)_{G} @=\C{H}(G)_{G} \\
@V\wt{a}_{P,G}^* VV @Va_{M,G}^*VV @V R_M^G VV\\
 \C{H}(\wt{G}_P)_{G} @>|\Dt_{P,G}|\cdot a_{M,P}^*>> \C{H}(\wt{G}_M)^{\reg}_G @>\varphi_M^G>> \C{H}(M)^{\reg/G}_M\\
@V\varphi_P^G VV @V\varphi_P^GVV @|\\
 \C{H}(P)_P @>|\Dt_{P,G}|\cdot a_{M,P}^*>>\C{H}(\wt{P}_M)^{\reg/G}_P @>\varphi_M^P>> \C{H}(M)^{\reg/G}_M.
\end{CD}
\end{equation}
We claim that all inner squares of \form{gl3} are commutative. Indeed, the top right square is commutative by the definition of $R_M^G$, the bottom left square is commutative by the functoriality of $\varphi_P^G$, and the bottom right square is commutative by the equality $\varphi_M^P\circ \varphi_P^G=\varphi_M^G$.

Finally, the commutativity of the top left square of \form{gl3} follows from the equality $a_{M,G}=a_{M,P}\circ a_{P,G}$ (see \re{simple}(d))
and \rl{meas}.
\end{proof}

\section{Normalized induction, independence of $\P$, and stability}

\begin{Emp} \label{E:normcoind}
{\bf Normalized restriction and induction.}

(a) Recall that in the construction of the restriction map $\Res_{P;M}^G$ in \re{main}(c) we used the
isomorphism $\C{O}_{\wt{\G}}\otimes_F V_\G\isom \C{K}_{\wt{\G}}:v\mapsto \om_{\wt{\G}}(v)$
(see \re{canbun}(c)). Instead we could use the isomorphism $\C{O}_{\wt{\G}}\otimes_F V_\G\isom \C{K}_{\wt{\G}}:v\mapsto \om_{\wt{\G}}^r(v)$ (see \re{right}(b)). Since  $\om_{\wt{\G}}^r(v)=\la_{\wt{\G}}\cdot\om_{\wt{\G}}(v)$, the resulting restriction map would be $|\la_P|\cdot  \Res_{P;M}^G$.

(b) We denote by $r_{P;M}^G:\C{H}(G)_G\to \C{H}(M)_M$, the map $|\la_P|^{1/2}\cdot  \Res_{P;M}^G$, and call it
the {\em normalized restriction map}. Let  $i_{P;M}^{G}:\wh{C}^{M}(M)\to \wh{C}^{G}(G)$ be the dual map of
$r_{P;M}^G$, called the {\em normalized induction map}. Explicitly, $i_{P;M}^{G}(\chi)=\Ind_{P;M}^G(|\la_P|^{1/2}\cdot\chi)$ for every $\chi\in\wh{C}^{M}(M)$.

(c) For an admissible representation $\rho$ of $M$, we denote by
$\un{i}_{P;M}^G(\rho)$ the representation $\un{\Ind}_{P;M}^G(\rho\otimes|\la_P|^{1/2})$ of $G$ and call it {\em the normalized parabolic induction}.
\end{Emp}

\begin{Emp}
{\bf Remark.}  If $\G$ is semisimple and simply connected, and $\P$ is a Borel subgroup of $\G$, then the normalized restriction map $r_{P;M}^G$ has a geometric interpretation.

Indeed, in this case, the homomorphism
$\la_{\P}:\P\to\B{G}_m$ has a unique square root $\la_{\P}^{1/2}:\P\to\B{G}_m$. Furthermore, $\la_{\P}^{1/2}$ gives rise to a morphism
$\la_{\wt{\G}}^{1/2}:\wt{\G}\to\B{G}_m$ (by \re{simple}(a), hence to an isomorphism $\C{O}_{\wt{\G}}\otimes_F V_\G\isom \C{K}_{\wt{\G}}:v\mapsto \la_{\wt{\G}}^{1/2}\cdot\om_{\wt{\G}}(v)$, and $r_{P;M}^G$ is obtained from this isomorphism by the construction \re{main}.
\end{Emp}

The following result follows immediately from \rco{ind}.

\begin{Cor} \label{C:ind'}
We have the equality of characters $\chi_{\un{i}_{P;M}^G(\rho)}=i_{P;M}^G(\chi_\rho)$.
\end{Cor}

Next, we write a version of \rl{res} for the normalized restriction.

\begin{Cor} \label{C:resnorm}
In the case of \re{setup},  the following diagram is commutative:
\begin{equation} \label{Eq:genlevi'}
\begin{CD}
\C{H}(G)_{G} @>R_H^G>> \C{H}(H)^{\reg/G}_H\\
  @Vr_{P;M}^GVV   @A|\Dt_{M,G}|^{1/2}\cdot {\res}AA\\
\C{H}(M)_{M}  @>R_H^M>> \C{H}(H)^{\reg/M}_H.
\end{CD}
\end{equation}
\end{Cor}

\begin{proof}
By \rcl{parab}, we have the identity $|\Dt_{M,G}|^{1/2}=|\Dt_{P,G}|\cdot|\la_P|^{-1/2}$, so
the assertion follows from \rl{res}.
\end{proof}

\begin{Lem} \label{L:inject}
The restriction map $\res: \C{H}(M)_M\to \C{H}(M)^{\reg/G}_M$ from \re{notrestr}(b) is injective.
\end{Lem}
\begin{proof}
By definition, the connected center $\Z(\M)^0\subset \Z(\M)$ acts on $\Lie \G/\Lie \M$ by a direct sum of non-trivial characters.
Therefore for every $m\in \M$, the locus of $z\in \Z(\M)$ such that $zm\in \M^{\reg/\G}$, is open and Zariski dense.
Similarly, the action of $\Z(\M)$ on $\M$ induces an action of $\Z(\M)$ on $\c_\M$, and for every $m\in \c_\M$, the locus of $z\in \Z(\M)$ such that $z(m)\in \c_\M^{\reg/\G}$, is open and Zariski dense. Thus, for every open subgroup $U\subset Z(M)$ we have $U(c_M^{\reg/G})=c_M$.

Note that $Z(M)$ acts smoothly on $\C{H}(M)$ and induces a smooth action on $\C{H}(M)_M$. Fix a non-zero $h\in \C{H}(M)_M$, and let $U\subset Z(M)$ be the  stabilizer of $h$. Since $U(c_M^{\reg/G})=c_M$, we conclude that $h|_{c_M^{\reg/G}}\neq 0$, thus $\res(h)\neq 0$.
\end{proof}

Now we show that the normalized induction map does not depend on $\P$.

\begin{Cor} \label{C:unit}
(a) The normalized restriction map $r_{P;M}^G:\C{H}(G)_G\to\C{H}(M)_M$ and the normalized induction map
$i_{P;M}^G:\wh{C}^{M}(M)\to\wh{C}^{G}(G)$ do not depend on $\P$.

(b) For every admissible representation $\rho$ of $M$, the set of composition factors of $\un{i}_{P;M}^G(\rho)$ does not depend on
$\P$.
\end{Cor}

\begin{proof}
(a) It suffices to show the assertion for $r_{P;M}^G$. By \rl{inject}, it suffices to show that the composition
$\res\circ r_{P;M}^G:\C{H}(G)_{G}\to \C{H}(M)^{\reg/G}_M$ does not depend on $\P$.
On the other hand, by \rco{resnorm} for $\H=\M$, this composition equals $|\Dt_{M,G}|^{-1/2}\cdot R_M^G$.

(b) It is enough to show that the character of $\un{i}_{P;M}^G(\rho)$ does not depend on
$\P$. But this follows from (a) and \rco{ind'}.
\end{proof}

\begin{Emp}
{\bf Notation.}
By \rco{unit}(a), we can now denote  $r_{P;M}^G$ and $i_{P;M}^G$, simply by $r_{M}^G$ and $i_{M}^G$, respectively.
\end{Emp}

 \begin{Emp} \label{E:orb}
{\bf (Stable) orbital integrals.} Let $\G^{\rss}\subset \G$ be the set of regular semisimple elements.

(a) Apply the notation of \re{restr} when $\H=\S\subset \G$ is a maximal torus defined over $F$.
In this case, $\c_{\S}^{\reg/\G}$ equals $\S^{\rss}_\G:=\S\cap \G^{\rss}$, and $R_S^G$ is the map $\C{H}(G)\to\C{M}^{\infty}(S^{\rss}_G)$.
We denote by $O_{S;G}:\C{M}^{\infty}(S^{\rss}_G)^*\to \wh{C}^{G}(G)$ the dual of $R_S^G$ and call it the {\em orbital integral} map.

(b) Explicitly, let $\pr:(\G/\S)\times \S_\G^{\rss}\to  \S_\G^{\rss}$ be the projection.
Since $\S$ acts on $\S^{\rss}_\G$ trivially, it follows from \re{ind}(d) that
$R_S^G$ equals the composition
\[\C{H}(G)_{G}\overset{a_{S,G}^*}{\lra} \C{H}([(G/S)\times S_G^{\rss}]/S_G^{\rss})_{G}\overset{\pr_!}{\lra} \C{M}^{\infty}(S_G^{\rss}).
\]

(c) Denote by $(R^S_G)^{\st}:\C{H}(G)_{G}\to \C{M}^{\infty}(S_G^{\rss})$ the composition
\[\C{H}(G)_{G}\overset{a_{S,G}^*}{\lra} \C{H}([(\G/\S)(F)\times S_G^{\rss}]/S_G^{\rss})_{G}\overset{\pr_!}{\lra} \C{M}^{\infty}(S_G^{\rss}),
\]
and let  $O^{\st}_{S;G}:\C{M}^{\infty}(S_G^{\rss})^*\to \wh{C}^{G}(G)$ be the dual map, called
the {\em stable orbital integral.}
\end{Emp}

\begin{Emp} \label{E:compar}
{\bf Comparison.}
(a) Let $[g_1],\ldots,[g_n]\in (\G/\S)(F)$ be a set of representatives of the set of $G$-orbits.
For every $j$, let $\S_j\subset \G$ be the stabilizer of $[g_j]$. Then $\S_j$ is a maximal torus
of $\G$, and there is a canonical isomorphism $i_j:\S\isom \S_j$. Explicitly, if
$g_i\in \G(\ov{F})$ is any representative of $[g_i]$, then $i_j$ is the map $s\mapsto g_jsg^{-1}_j$.
By construction, we have the equality
%$(R^S_G)^{\st}=\sum_{j+1}^n i_j^*\circ R^{S_j}_G:\C{H}(G)_{G}\to \C{M}^{\infty}(S_G^{\rss})$,
$O^{\st}_{S;G}=\sum_{j=1}^n O_{S_j;G}\circ (i_j)_*$.

(b) Fix a Haar measure $|\om_S|$ on $S$. Then for every $\gm\in S^{\rss}_G$, we can consider a ``$\dt$-function" $\dt_{\gm}\in\C{M}^{\infty}(S_G^{\rss})^*$, defined by the formula $\dt_{\gm}(f|\om_S|)=f(\gm)$ for every
$f\in C^{\infty}(S_G^{\rss})$. Then the construction \re{orb} gives us ``classical" (stable) orbital integrals
$O_{\gm,G}:=O_{S,G}(\dt_{\gm})\in \wh{C}^{G}(G)$ (resp. $O^{\st}_{\gm,G}:=O^{\st}_{S,G}(\dt_{\gm})\in  \wh{C}^{G}(G)$).
For example, observation (a) implies that a stable orbital integral is a sum of orbital integrals.
\end{Emp}

\begin{Emp} \label{E:orbapl}
{\bf Application.} Let $U\subset S^{\rss}_G$ be a dense subset, and let $h\in\C{H}(G)_G$ be such that $O_{\gm,G}(h)=0$ for every $\gm\in U$.
Then  $O_{\gm,G}(h)=0$ for every $\gm\in S^{\rss}_G$.

Indeed, let $f\in C^{\infty}(S_G^{\rss})$ such that $R_G^S(h)=f|\om_S|$ (see \re{compar}(b)).
Then for every $\gm\in S^{\rss}_G$, we have $O_{\gm,G}(h)=f(\gm)$. Hence, by assumption, we have $f(\gm)=0$ for all $\gm\in U$. Thus $f=0$, since
$f$ is locally constant and $U\subset S^{\rss}_G$ is dense.
\end{Emp}

\begin{Cor} \label{C:normorb}
Let $\M\subset \G$ be a Levi subgroup, and let $\S\subset \M$ be a maximal torus defined over $F$. Then
$\S\subset \G$ is a maximal torus, and the following diagram is commutative (compare \re{orb}):
\[
\begin{CD}
\C{M}^{\infty}(S^{\rss}_G)^* @>O^{\st}_{S,G}>> \wh{C}^{G}(G) \\
@V{(|\Dt_{M,G}|^{1/2}\cdot\res)^*}VV                                          @Ai_{M}^G AA \\
\C{M}^{\infty}(S^{\rss}_M)^* @>O^{\st}_{S,M}>>  \wh{C}^{M}(M),
\end{CD}
\]
and similarly for $O_{S,G}$ and $O_{S,M}$.
\end{Cor}

\begin{proof}
The assertion for orbital integrals is simply the dual of \rco{resnorm}
for $\H=\S$. The assertion for stable orbital integrals follows from that for orbital
integrals, by the observation \re{compar}(a) and the equality $(\G/\M)(F)=G/M$.
\end{proof}

\begin{Emp} \label{E:stable}
{\bf Stable generalized functions and representations.}

(a) We denote by $\wh{C}^{\st}(G)\subset \wh{C}^{G}(G)$ the closure of the image of
\[
O^{\st}_{G}:=\oplus_{\S\subset \G} O^{\st}_{S,G}:\oplus_{\S\subset \G} \C{M}^{\infty}(S_G^{\rss})^*\to \wh{C}^{G}(G),
\]
where the sum is taken over the set of all maximal tori $\S\subset\G$ defined over $F$.
Elements of $\wh{C}^{\st}(G)$ are called {\em stable generalized functions}.

(b) An admissible representation $\pi$ of $G$ is called {\em stable}, if its character $\chi_{\pi}$ is stable.
\end{Emp}

%The first corollary of \rp{orb} says that the induction preserves stability.

\begin{Cor} \label{C:stab}
(a) The induction map $i_{M}^G:\wh{C}^{M}(M)\to\wh{C}^{G}(G)$ sends stable generalized functions
to stable ones.

(b) The induction functor $\un{i}_{P;M}^G$ sends stable representations to stable ones.
\end{Cor}

\begin{proof}
(a) Since $S^{\rss}_G\subset S^{\rss}_M$ is dense, the restriction map
$\res:\C{M}^{\infty}(S^{\rss}_M)\to \C{M}^{\infty}(S^{\rss}_G)$ is injective. Thus the dual map
$\res^*$ is surjective. Hence, by  \rco{normorb}, we get an inclusion $i_{M}^G (\im O^{\st}_{S,M})\subset \wh{C}^{\st}(G)$ for every $\S\subset\M$. Since $i_{M}^G$ is continuous, the assertion follows.

(b) Follows from (a) and \rco{ind'}.
\end{proof}

%\begin{Emp} \label{E:rem orb}
%{\bf Remark.} (a) In \re{orb} we defined (stable) orbital integrals separately for each maximal torus $S$.
%Since the maximal tori lie in a family, this construction can be made uniform.

%(b) Alternatively, we could make use of the Chevalley morphism $\nu_G:G\to \c_G$. This would simplify the construction,
%if one is only interested in stable orbital integrals for strongly regular elements, but is less convenient for orbital integrals, for example.
%\end{Emp}

%Since the subset $\bO _M U_0\subset P_0$ is Ad-invariant we can consider the quotient $Y$ of $G\times bO _M U_0/P_0$ where $P_0$ acts  by
%left shifts on $G$ and by conjugation on $\bO _M U_0$. Let $\ti h:G\times \bO _M U_0\to G$ by given by $\ti h(g,a)=g^{-1}ag$.
%It is easy to see that $\ti h$ induces an isomorphism $Y\to \bO _G$

%Lemma 2 follows immediately from this equality.

\appendix\section{A generalization of a theorem of Lusztig--Spaltenstein}

%As a consequence, we get the following purely algebraic result, shown for algebraically closed fields by Lusztig and Spaltenstein \cite{LS}. First we introduce a notation.

\begin{Emp} \label{E:unipclas}
{\bf Notation.} Let $F$ be an infinite field. All algebraic varieties and all morphisms of algebraic varieties are over $F$.

(a) Let $\G$ be a connected reductive group, $\P\subset \G$  a parabolic subgroup, $\U\subset \P$ the unipotent radical, and $\M\subset \P$ a Levi subgroup.

(b) For an algebraic variety $\X$, we denote the set $\X(F)$ by $X$. In particular, we have $G=\G(F)$, $\wt{G}_P=\wt{\G}_{\P}(F)$, etc. (compare \re{alcase}).

(c) For an $\Ad P$-invariant subset $D\subset  P$, we set $\Ind_P^G(D):=G\overset{P,\Ad}{\times}D\subset\wt{G}_P$.

(d) For an $\Ad M$-invariant subset $C\subset  M$, we set $\Ind_M^G(C):=G\overset{M,\Ad}{\times}C\subset \wt{G}_M$, $C_P:=U\cdot C\subset P$, $\Ind_P^G(C_P)\subset\wt{G}_P$ and
$C_{P;G}:=a_{P,G}(\Ind_P^G(C_P))\subset G$, where $a_{P,G}:\wt{G}_P\to G$ was defined in \re{grspr}.
\end{Emp}

From now on, we assume that $C\subset  M$ is a unipotent $M$-conjugacy class.

\begin{Emp} \label{E:qind}
{\bf Question.} Does the set $C_{P;G}$ depend on the choice of $\P\supset \M$?
\end{Emp}

\begin{Emp} \label{E:remls}
{\bf Remarks.} (a) $C_{P;G}$ is a union of unipotent conjugacy classes
in $G$.

(b) Let $F$ be algebraically closed. By a theorem of Chevalley, $C_{P;G}\subset G$ is a constructible set, whose Zariski closure $\ov{C}_{P;G}$ is irreducible. This case was considered by Lusztig and Spaltenstein in \cite{LS}, and they showed that $\ov{C}_{P;G}$ does not depend on  $\P$, using representation theory. A simpler proof of this fact was given later by Lusztig \cite[Lem 10.3(a)]{Lu}.
\end{Emp}

The goal of this appendix is to generalize the result of \cite{LS} to other fields.

\begin{Emp} \label{E:rat}
{\bf Saturation.} Let $\X$ be an algebraic variety over $F$, and  let $A\subset X$ be a subset.

(a) We denote by $\sat'(A)=\sat'_\X(A)\subset X$ the union $\cup_{(\V,x,f)}f(x)$, taken over triples $(\V,x,f)$, where
$\V\subset\B{A}^1$ is an open subvariety, $x\in V$, $\V':=\V\sm \{x\}$, and $f:\V\to \X$ is a morphism such that $f(V')\subset A$.

(b) We say that a subset $A\subset X$ is {\em saturated}, if $\sat'(A)=A$.

(c) Let $\sat(A)\subset X$ be the smallest saturated subset, containing $A$.
\end{Emp}

%The goal of this appendix is to prove the following result.

\begin{Thm} \label{T:LS}
The saturation $\sat(C_{P;G})$ does not depend on $\P$.
\end{Thm}

\begin{Emp} \label{E:ratrem}
{\bf Remarks.} (a) The notion of saturation is only reasonable, if the variety $\X$ is rationally connected.

(b) For every closed subvariety $\Y\subset\X$, the subset $\Y(F)\subset X$ is  saturated. Also, if $F$ is a local field, then every closed subset of $X$  is saturated.

(c) If $\X=\B{A}^1$, then a subset $A\subset X$ is saturated if and only if either $A=X$ or $X\sm A$ is infinite.

(d) By (c), saturated subsets of $X$ are not closed under finite unions. Therefore the set $X$ does not have a topology, whose closed subsets are saturated subsets. On the other hand, our proof of \rt{LS} indicates that in some respects saturated sets behave like closed subsets in some topology.
%On the other hand, the results of this section indicate that in many respects saturated sets behave like closed subsets in some topology.
% We wonder whether saturated subsets of $X$ are closed under finite unions, when $F$ is Hilbertian?
%(c) We are wondering whether for a rationally connected variety there exits a topology on $X$ such that the arguments of this section would remain correct if one replaces saturated subsets by closed subsets?
\end{Emp}

\begin{Lem} \label{L:propcl}
Let $\X$ and $\Y$ be algebraic varieties.

(a) For a morphism $f:\X\to \Y$ and a subset $A\subset X$, we have an inclusion
$f(\sat'(A))\subset \sat(f'(A))$.

%(b) Let $Y\subset X$ be an open subset and $A\subset X$. Then $\sat'_X(A)\cap Y=\sat'_Y(A\cap Y)$.

(b) For every $A\subset X$ and $B\subset Y$, we have %the equality
$\sat'(A\times B)=\sat'(A)\times \sat'(B)$.

(c) Let $\H$ be an algebraic group, and let $f:\X\to \Y$ be a principal $\H$-bundle, locally trivial in the Zariski topology.
Then for every subset $A\subset Y$ we have the equality $\sat'(f^{-1}(A))=f^{-1}(\sat'(A))$.

(d) For an $\Ad P$-invariant subset $A\subset P$, the corresponding
subset $\Ind_P^G(A)\subset\Ind_P^G(X)$ satisfies
$\sat'(\Ind_P^G(A))=\Ind_P^G(\sat'(A))$.

(e) Let $\Y\subset \X=\B{A}^n$ be an open dense subvariety. Then $\sat'_X (Y)=X$.
\end{Lem}

\begin{proof}
(a) is clear.

%(b) Since $Y\subset X$, for every $a\in\sat'_X(A)\cap Y$ there exists triple $(V,x,f)$ in the definition
%of $\sat'_X$ such that $f(x)=a\in Y$, thus we can shrink $V$ so that $f(V)\subset Y$.

(b) The inclusion $\subset$ follows from (a). Conversely, assume that $a\in \sat'(A)$ and $b\in \sat'(B)$ are defined using triples
$(\V_a,x_a,f_a)$ and $(\V_b,x_b,f_b)$, respectively, where $\V_a$ and $\V_b$ are open subsets of $\B{A}^1$.
Then we can assume that $\V_a=\V_b\subset\B{A}^1$ and $x_a=x_b$, which implies that
$(a,b)\in \sat'(A\times B)$.

(c) Since the saturation $\sat'$ is local in the Zariski topology, we can assume that
$\X=\Y\times\H$. In this case the assertion follows from (b).

(d) Arguing as in \re{ind}(b), the assertion follows from (c).

(e) It suffices to show that for every $x\in\B{A}^n(F)$, there exists a line $\L\subset \B{A}^n$,
defined over $F$, such that $x\in\L$ and $\L\cap \Y\neq \emptyset$. Consider the variety $\B{P}_x$ of
lines $\L\subset \B{A}^n$ such that $x\in \L$. Since $\Y\subset \B{A}^n$ is Zariski dense, the set of
$\L\in\B{P}_x$ such that $\L\cap \Y\neq \emptyset$, is Zariski dense. Since
$\B{P}_x\cong\B{P}^{n-1}$, while $F$ is infinite, the subset $\B{P}_x(F)\subset\B{P}_x$ is Zariski dense, and the
assertion follows.
\end{proof}

\begin{Emp}
{\bf Remark.}
All the properties of $\sat'$, formulated in \rl{propcl}, have natural analogs for $\sat$.
\end{Emp}

\begin{Emp} \label{E:rel}
{\bf Relative saturation.} Let $h:\X\to \Y$ be a morphism and $A\subset X$.

(a) Denote by $\sat'(h;A)\subset \sat'(h(A))$ the union $\cup_{(\V,x,f)}f(x)$, taken over all triples
$(\V,x,f)$ in the definition of $\sat'(h(A))$ (see \re{rat}(a)) such that $f|_{\V'}:\V'\to \Y$ has a lift
$\wt{f}':\V'\to \X$ with $\wt{f}'(V')\subset A$.

(b) If $h$ is proper, then $\sat'(h;A)= h(\sat'(A))$. Indeed, the valuative criterion implies that
every pair $(f,\wt{f}')$ as in (a) defines a unique morphism $\wt{f}:\V\to \X$ such that $h\circ \wt{f}=f$ and
$\wt{f}|_{\V'}=\wt{f}'$.

(c) Let $\X'\subset\X$ be an open subvariety such that $A\subset X'$, and let $h':=h|_{\X'}:\X'\to \Y$ be the restriction.
By definition, $\sat'(h';A)=\sat'(h;A)$.
\end{Emp}

\begin{Emp} \label{E:prep}
{\bf Notation}.
(a) Let $a_{\M,\G}:\wt{\G}_\M^{\reg}\to \G$  be the map defined in \re{grspr} and \re{simple}, let
$\nu_{\G}:\G\to \c_\G$ be the Chevalley map (see \re{chev}(a)), and set $e_\G:=\nu_\G(1)\in \c_\G$.

(b) Let $\G^{\der}\subset\G$ be the derived group of $\G$, $\Z(\M)$ the center of $\M$, and set $\Z_\M:=(\Z(\M)\cap \G^{\der})^0$.
Then $\Z_\M$ is a split torus over $F$.  Set $\Z^{\reg}_{\M}:=\Z_\M\cap \M^{\reg/\G}$. Notice that since
$\Z_\M$ acts on $\Lie \G/\Lie \M$ by a direct sum of non-trivial characters, the subset
$\Z^{\reg}_{\M}\subset \Z_\M$ is open and dense.

(c) Set $C^{\reg}_Z:=C\cdot Z^{\reg}_M\subset M$. Since
$C\subset M$ consists of unipotent elements, and $Z_M\subset Z(M)$, we have $C^{\reg}_Z\subset M^{\reg/G}$.
Also $C^{\reg}_Z$ is $\Ad M$-invariant,
so we can form a subset $\Ind_M^G(C^{\reg}_Z)\subset \wt{G}_M^{\reg}$.

(d) Set $C^{\reg}_{P,Z}:=p^{-1}(C^{\reg}_Z)=C_P\cdot Z^{\reg}_M\subset P^{\reg/G}$ (see \rl{eqrk}(d)).
\end{Emp}

\begin{Emp}
\begin{proof}[Proof of \rt{LS}]
Consider the subset $D_P:=a_{P,G}(\sat'(\Ind_P^G(C_P)))$ of $G$. Since $C_{P;G}=a_{P,G}(\Ind_P^G(C_P))$, we have inclusions
$C_{P;G}\subset D_P\subset \sat(C_{P;G})$ (see \rl{propcl}(a)), thus  $\sat(D_P)=\sat(C_{P;G})$. It suffices to show
that $D_P$ does not depend on $\P$. But this follows from the following description of $D_P$.
\end{proof}
\end{Emp}

\begin{Cl} \label{C:form}
We have the equality $D_P=\sat'(a_{M,G};\Ind_M^G(C^{\reg}_Z))\cap \nu_G^{-1}(e_\G)$.
\end{Cl}

\begin{proof}
Recall (see \re{simple}(d)) that morphism $a_{M,G}$ factors as $\wt{G}_M^{\reg}\overset{a_{M,P}}{\lra}\wt{G}_P\overset{a_{P,G}}{\lra} G$.
Notice that $a_{M,P}:\wt{G}_M^{\reg}\to \wt{G}_P$ is an open embedding (use \rl{eqrk}(e)) and it satisfies
$a_{M,P}(\Ind_M^G(C^{\reg}_Z))=\Ind_P^G(C^{\reg}_{P,Z})$. Therefore, by \re{rel}(c), we have the equality
\[
\sat'(a_{M,G};\Ind_M^G(C^{\reg}_Z))=\sat'(a_{P,G};\Ind_P^G(C^{\reg}_{P,Z})).
\]
Next, since $a_{P,G}$ is proper, we conclude from \re{rel}(b) that  \[
\sat'(a_{P,G};\Ind_P^G(C^{\reg}_{P,Z}))=a_{P,G}(\sat'(\Ind_P^G(C^{\reg}_{P,Z}))).
\]
Thus it suffices to show the equality
\[
a_{P,G}(\sat'(\Ind_P^G(C_P)))= a_{P,G}(\sat'(\Ind_P^G(C^{\reg}_{P,Z})))\cap \nu_G^{-1}(e_\G).
\]
Using the commutative diagram from \re{simple}(b) for $\H=\P$ and equality $\pi_{P,G}^{-1}(e_\G)=e_\P$ (see \rl{eqrk}(f)),
we conclude that $a_{P,G}(A)\cap \nu_G^{-1}(e_\G)=a_{P,G}(A\cap \nu_{\wt{G}_P}^{-1}(e_\P))$ for every subset $A\subset \wt{G}_P$.
Thus it suffices to show the equality
\[
\sat'(\Ind_P^G(C^{\reg}_{P,Z}))\cap \nu_{\wt{G}_P}^{-1}(e_\P)= \sat'(\Ind_P^G(C_P))\subset\wt{G}_P.
\]
Using \rl{propcl}(d), it suffices to show the equality
\begin{equation} \label{Eq:sat}
\sat'(C^{\reg}_{P,Z})\cap \nu_P^{-1}(e_\P)= \sat'(C_P)\subset P.
\end{equation}
Set $\P_{\unip}:=\nu_\P^{-1}(e_\P)\subset \P$, and
$\P_{\Z_\M}:=\nu_\P^{-1}(\nu_\P(\Z_\M))\subset \P$. Since the map $\nu_\P|_{\Z_\M}:\Z_\M\to \c_\P\cong \c_\M$ is a closed embedding, the multiplication map induces an isomorphism $\P_{\unip}\times \Z_\M\isom \P_{\Z_\M}$. Moreover, it induces a
bijection $C_P\times Z^{\reg}_M\isom C^{\reg}_{P,Z}$. Thus, formula \form{sat} follows from the equality
\[
\sat'_{\P_{\unip}\times \Z_\M}(C_P\times Z^{\reg}_M)=\sat'_{\P_{\unip}}(C_P)\times \sat'_{\Z_\M}(Z^{\reg}_M)=\sat'_{\P_{\unip}}(C_P)\times Z_M,
\]
which follows from \rl{propcl}(b),(e).
\end{proof}

\begin{Cor} \label{C:LS}
(a) If $F$ is algebraically closed, then the closure $\cl(C_{P;G})\subset G$ of $C_{P;G}$ in the Zariski topology does not depend on $\P$.

(b) If $F$ is a local field, then the closure $\cl(C_{P;G})\subset G$ of $C_{P;G}$ in the analytic topology does not depend on $\P$.
\end{Cor}

\begin{proof}
In both cases, every closed subset in $G$ is saturated. Therefore we have inclusions
$C_{P;G}\subset \sat(C_{P;G})\subset\cl(C_{P;G})$, which imply that $\cl(C_{P;G})=\cl(\sat(C_{P;G}))$.
Thus the assertion follows from \rt{LS}.
\end{proof}

\begin{Emp} \label{E:heart}
{\bf Notation.} For an $\Ad G$-invariant subset $D\subset G$, we denote by $D^{\heartsuit}\subset D$ the union of $G$-conjugacy classes that are Zariski dense in (the Zariski closure of) $D$.
\end{Emp}

\begin{Emp} \label{E:qind}
{\bf Question.} Is it true that $C^{\heartsuit}_{P;G}$ is independent of $\P$?
\end{Emp}

\begin{Emp}
{\bf Remarks.} (a)  Let $F$ be algebraically closed. Since the number of unipotent conjugacy classes in $G$ is finite, we conclude that $C^{\heartsuit}_{P;G}$ is a single conjugacy class. %Thus $C^{\heartsuit}_{P;G}$ is independent of $\P$ (compare \re{remls}(b)).

(b) Let $F$ be general. Then, by (a), $C^{\heartsuit}_{P;G}$ is a union of unipotent conjugacy classes, belonging to a single conjugacy class over $\ov{F}$.
\end{Emp}

\begin{Lem} \label{L:op}
Let $F$ be either algebraically closed or local. Then for every $\Ad G$-invariant subset $D\subset G$,
we have $D^{\heartsuit}=\sat(D)^{\heartsuit}$.
\end{Lem}

\begin{proof}
Let $\cl(D)\subset G$ be the closure of $D$ in the Zariski topology if $F$ is algebraically closed, and in the analytic topology if $F$ is local.
Then, as in the proof of \rco{LS}, we have $D\subset \sat(D)\subset \cl(D)$. Thus, it suffices to show that $D^{\heartsuit}=\cl(D)^{\heartsuit}$.

Let $O\subset \cl(D)$ be a Zariski dense $G$-conjugacy class, and let $\D\subset \G$ be the Zariski closure of $D$.
Choose $x\in O$. Then the morphism $\G\to \D:g\mapsto gxg^{-1}$ is dominant. Therefore, in both cases, the corresponding
map $G\to \cl(D)$ is open. Thus $O\subset\cl(D)$ is open, hence $O\subset D$.
\end{proof}

\begin{Cor} \label{C:LS2}
Let $F$ be either algebraically closed or local. Then the subset $C^{\heartsuit}_{P;G}\subset G$ does not depend on $\P$.
\end{Cor}

\begin{proof}
This follows immediately from \rt{LS} and \rl{op}.
\end{proof}

\begin{Emp}
{\bf Remark.}
%(a) If one is only interested in \rco{LS} or \rco{LS2}, the argument could be simplified. Namely, instead of considering saturation, one could consider the closure in the Zariski or natural topology as in the proof of \rl{op}.
We do not expect that the conclusion  \rl{op} holds for an arbitrary field $F$. We wonder whether the equality
$\sat(C_{P;G})^{\heartsuit}=C^{\heartsuit}_{P;G}$ always holds.
\end{Emp}

\section{On a theorem of Harish-Chandra}

\noindent The goal of this section is to explain the proof of the following result, usually attributed to Harish-Chandra.

\begin{Thm} \label{T:HC}
Let $F$ be a local non-archimedean field of characteristic zero, and let $h\in\C{H}(G)_G$  be such that $O_{\gm}(h)=0$ for every
$\gm\in G^{\rss}$. Then $h=0$.
\end{Thm}

\begin{Emp} \label{E:remHC}
{\bf The Lie algebra analog.} Let $\frak{g}$ be the Lie algebra of $\G$, equipped with  the adjoint action of $G$, and let $h\in\C{H}(\frak{g})_G$  be such that $O_{x}(h)=0$ for every $x\in \frak{g}^{\rss}$. Then the original theorem of Harish-Chandra (\cite[Thm 3.1]{HC}) asserts that $h=0$.
\end{Emp}

The goal of this section is to deduce \rt{HC} from its Lie algebra analog.

\begin{Emp} \label{E:gdomain}
{\bf $G$-domains.} (a) Let $X$ be a smooth analytic variety over $F$ equipped with an action of $G$, let $\C{H}(X)$ be the
space of locally constant measures with compact support (see \re{smmeas}) and let $\C{H}(X)_G$ be the space of
$G$-coinvariants.

(b) By a {\em $G$-domain} in $X$, we mean an  open and closed $G$-invariant subset $U\subset X$.
Then $\C{H}(U)\subset\C{H}(X)$ is a $G$-invariant subspace, and the map $h\mapsto 1_U\cdot h$ is a $G$-equivariant projection $\C{H}(X)\to\C{H}(U)$.
Taking $G$-coinvariants, we get an inclusion  $\C{H}(U)_G\hra\C{H}(X)_G$ and a projection  $\C{H}(X)_G\to\C{H}(U)_G\subset\C{H}(X)_G:h\mapsto h|_U$.
\end{Emp}

\begin{Lem} \label{L:inj}
Let $f:X\to Y$ be a proper, surjective $G$-equivariant local isomorphism  between smooth analytic varieties.
Then the pullback map $f^*:\C{H}(Y)_G\to\C{H}(X)_G$ (see \re{pullback}(b)) is injective.
\end{Lem}

\begin{proof}
For every $m\in\B{N}$, we denote by $Y_m\subset Y$ the set of all $y\in Y$ such that the cardinality of $f^{-1}(y)$ is $m$.
The assumptions on $f$ imply that every $Y_m\subset Y$ is a $G$-domain, and that $Y$ is the disjoint union of the $Y_m$'s.
Then every $X_m:=f^{-1}(Y_m)\subset X$ is a $G$-domain as well, and it suffices to show that the induced map $f^*:\C{H}(Y_m)_G\to\C{H}(X_m)_G$  is injective.
Since for every $h\in\C{H}(Y_m)$  we have $f_!f^*(h)=mh$, we are done.
\end{proof}

\begin{proof}[Proof of \rt{HC}]
We carry out the proof in five steps.

\vskip4truept

{\bf Step 1.}  There exists a $G$-domain $U\ni 1$ in $G$ such that $h|_U=0$.

\begin{proof}
Observe first that there exist $G$-domains $\frak{u}\ni 0$ in $\frak{g}$ and  $U\ni 1$ in $G$ such that
the exponential map induces an $\Ad G$-equivariant analytic isomorphism $\epsilon:\frak{u}\isom U$.
Namely, if $\G=\GL_n$, the assertion is straightforward, and the general case follows from it.

We claim that this $U$ satisfies the required property. Indeed, consider the pullback
$h':=\epsilon^*(h|_U)\in\C{H}(\frak{u})_G\subset \C{H}(\frak{g})_G$. It suffices to show
that $h'=0$. For $x\in \frak{g}^{\rss}$ we have $O_{x}(h')=0$ if $x\notin\frak{u}$, since
$h'\in\C{H}(\frak{u})_G$, and $O_{x}(h')=O_{\epsilon(x)}(h)=0$ if $x\in \frak{u}$ by our assumption on $h$.
Then $h'=0$ by \cite[Thm 3.1]{HC} (see \re{remHC}).
\end{proof}

{\bf Step 2.}  For every $s\in Z(G)$ there exists a $G$-domain $U\ni s$  in $G$ such that $h|_U=0$.

\begin{proof}
Since the map $g\mapsto gs:G\to G$ is $\Ad G$-equivariant, the assertion follows
from the $s=1$ case shown in Step 1. More precisely, if $U\ni 1$ is the $G$-domain constructed in Step 1, then $sU\ni s$
is the $G$-domain such that $h|_U=0$.
\end{proof}

{\bf Step 3.}  For every semisimple $s\in G\sm Z(G)$ there exists a $G$-domain $U\ni s$  in $G$  such that $h|_U=0$.

\begin{proof}
Let $\H:=\G_s^0$ be the connected centralizer of $s$. Then $\H\subsetneq \G$, and by induction, we may assume that \rt{HC} is valid
for $\H$.

Let $\H^{\reg/\G}\subset\H$ and $\c_{\H}^{\reg/\G}\subset\c_\H$ be the open subschemes defined in \re{reg}(b).
Note that $\H\subset \G$ is an equal rank subgroup (see \re{eqrk}(a)), and $s\in \H^{\reg/\G}$. Indeed, let $\T\ni s$ be a maximal torus of $\G$. Then
$s\in\T\subset\H$, and $\Z_{\G}(s)^0=\H=\Z_{\H}(s)^0$. Hence $s\in \H^{\reg/\G}$ by \re{eqrk}(b).

Let $\nu_{\H}:\H^{\reg/\G}\to \c_{\H}^{\reg/\G}$ be the Chevalley map (see \re{chev}(a) and \re{reg}(b)), and we denote by  $\nu_H:H^{\reg/G}\to c_{H}^{\reg/G}$ the induced map on $F$-points (compare \re{alcase}).

Choose an open and compact neigbourhood $V\subset c_{H}^{\reg/G}$ of $\nu_H(s)$, and consider its preimage $U':=\nu_{H}^{-1}(V)\subset H^{\reg/G}\subset H$.
Then $U'\subset H$ is an $H$-domain, so we can form the induced space $\Ind_H^G(U')$ (see \re{ind}) and the $G$-equivariant morphism
$f:=a_{H,G}|_{\Ind_H^G(U')}:\Ind_H^G(U')\to G:[g,x]\mapsto gxg^{-1}$ (compare \re{grspr}(c)).

Recall that the subset $\Ind_H^G(U')\subset \Ind_{\H}^{\G}(\H^{\reg/\G})(F)$ is open and closed (by \re{alg}(a)).
Since $a_{\H,\G}:\Ind_{\H}^{\G}(\H^{\reg/\G})\to \G$ is \'etale (see \rco{etale}), we conclude that $f$ is a local isomorphism.
On the other hand, since both morphisms $\iota_{\H,\G}:\Ind_{\H}^{\G}(\H^{\reg/\G})\to \G\times_{\c_\G} \c_\H^{\reg/\G}$ (see \rco{eqrank})
and $\pi_{\H,\G}:\c_{\H}\to\c_\G$ (see \re{chev}(b)) are finite, and $V\subset c_H$ is a compact subset, the composition
\[
f:\Ind_H^G(U')\overset{\iota_{H,G}}{\lra} G\times_{c_G} V \overset{\pi_{H,G}}{\lra}G
\]
is proper. Therefore $U:=\im f$ is a $G$-domain containing $s$, and we claim that $h|_U=0$.

By \rl{inj}, the induced map $f^*:\C{H}(U)_G\to\C{H}(\Ind_{H}^{G}(U'))_G$ is injective.
Let $\phi:\C{H}(U)_G\to\C{H}(U')_H$ be the composition of $f^*$ and the isomorphism $\varphi_H^G:\C{H}(\Ind_{H}^{G}(U'))_G\isom \C{H}(U')_H$ from \re{ind}(c).
Then $\phi$ is injective, thus it remains to show that $h':=\phi(h|_U)\in \C{H}(U')_H\subset \C{H}(H)_H$ is zero.

Since \rt{HC} is valid for $H$, it suffices to show that $0_{\gm}(h')=0$ for all $\gm\in H^{\rss}$.
This is clear for $\gm\notin U'$, since $h'\in \C{H}(U')_H$. By construction, for every $\gm\in U'\cap G^{\rss}$ we have $O_{\gm}(h')=O_{\gm}(h)$. Hence
$O_{\gm}(h')=0$ by the assumption on $h$. This shows that $0_{\gm}(h')=0$ for all $\gm\in H^{G-\rss}:=H\cap G^{\rss}$.

Finally, since $S\cap H^{G-\rss}\subset S\cap H^{\rss}$ is dense for every maximal torus $\S\subset \H$,
the equality $0_{\gm}(h')=0$ for every $\gm\in H^{\rss}$ now follows from \re{orbapl}.
\end{proof}

{\bf Step 4.} Let $U\subset G$ be a $G$-domain, and let $g=su$ be the Jordan decomposition of $g\in G$. Then $g\in U$ is and only if
$s\in U$.

\begin{proof}
Set $H:=G_s^0$. It suffices to show the closure of the
$\Ad H$-orbit of $u$ contains $1$, hence
the closure of the $\Ad G$-orbit of $g$ contains $s$.

Since $u$ is a unipotent element of $H$, the Zariski closure of the
$\Ad \H$-orbit of $u$ contains $1$. Hence the assertion follows from a theorem of Kempf
\cite[Cor. 4.3]{Ke}.
\end{proof}

{\bf Step 5: Completion of the proof.} By Steps 3 and 4, for every  $g\in G$ there exists a $G$-domain $U\ni g$  in $G$ such that $h|_U=0$.
From this the assertion follows. Indeed, choose a lift $\wt{h}\in\C{H}(G)$ of $h$, and
let $K\subset G$ be the support of $\wt{h}$. Since $K$ is compact, there is a finite collection of $G$-domains
$U_i,i=1,\ldots,n$ such that $K\subset \cup_{i}U_i$ and each $h_i:=h|_{U_i}$ is zero. Moreover,
replacing $U_j$ by $U_j\sm (\cup_{i=1}^{j-1}U_i)$ we can assume that the $U_i$'s are disjoint.
Then $h=\sum_{i=1}^n h_i=0$.
 \end{proof}

\end{document}